\documentclass[11pt]{amsart}
\usepackage{mathrsfs}
\usepackage{amsfonts}
\usepackage{graphicx,epsfig,amsmath,amsfonts}
\usepackage{color}
\usepackage{latexsym,amssymb}
\usepackage{lineno,version}
\usepackage{multirow}
\usepackage{bm}
\usepackage{float}
\usepackage{subfigure}
\catcode`\@=11
\theoremstyle{plain}
\@addtoreset{equation}{section}   

\@addtoreset{figure}{section}
\renewcommand\thefigure{\thesection.\@arabic\c@figure}

\newtheorem{theorem}{\indent Theorem}[section]
\newtheorem{lemma}{\indent Lemma}[section]

\newtheorem{remark}{\indent Remark}[section]

\newtheorem{example}{\indent Example}[section]
\newcommand{\ba}{\begin{array}}\newcommand{\ea}{\end{array}}
\newcommand{\be}{\begin{eqnarray}}\newcommand{\ee}{\end{eqnarray}}
\newcommand{\beq}{\begin{equation}}\newcommand{\eeq}{\end{equation}}
\newcommand{\bex}{\begin{eqnarray*}}
\newcommand{\eex}{\end{eqnarray*}}

\def\L{\mathcal L}

\font\tenbi=cmmib10   at 11 pt
\font\sevenbi=cmmib10 at 9pt
\font\fivebi=cmmib7 at 6pt
\newfam\bifam
\textfont\bifam=\tenbi \scriptfont\bifam=\sevenbi
\scriptscriptfont\bifam=\fivebi

\font\tendb=msbm10 at 12 pt
\font\sevendb=msbm7
\newfam\dbfam
\textfont\dbfam=\tendb \scriptfont\dbfam=\sevendb

\def\Dt{\Delta t}

\begin{document}
\title[Stochastic PDEs with random diffusion coefficients and Multiplicative noise]
{An efficient numerical approach for stochastic evolution PDEs driven by random diffusion coefficients and multiplicative noise$^*$}
\author[X. Qi, M. Azaiez, C. Huang, \& C.J. Xu]
{Xiao Qi$^{1}$
\quad
Mejdi Azaiez$^{1,2}$
\quad
Can Huang$^{1}$
\quad
Chuanju Xu$^{1,3}$}
\thanks{\hskip -12pt
${}^*$This research is partially supported by
NSFC grant 11971408, NSFC/ANR joint program ANR-16-CE40-0026-01, and the
French State in the frame of the ``Investments for the future" programme Idex Bordeaux
ANR-10-IDEX-03-02.\\
${}^{1}$School of Mathematical Sciences and
Fujian Provincial Key Laboratory of Mathematical Modeling and High Performance
Scientific Computing, Xiamen
University, 361005 Xiamen, China.\\
${}^{2}$Bordeaux INP, Laboratoire I2M UMR 5295, 33607 Pessac, France.\\
${}^{3}$Corresponding author. Email: cjxu@xmu.edu.cn}



\maketitle

\begin{abstract}
In this paper, we investigate the stochastic evolution equations (SEEs) driven by $\log$-Whittle-Mat$\acute{{\mathrm{e}}}$rn (W-M) random diffusion coefficient field and $Q$-Wiener multiplicative force noise. First, the well-posedness of the underlying equations is established 
by proving the existence, uniqueness, and stability of the mild solution. 
A sampling approach called approximation circulant embedding with padding is proposed to sample the random coefficient field. Then a spatio-temporal discretization method based on semi-implicit Euler-Maruyama scheme and finite element method
is constructed and analyzed. An estimate for the strong convergence rate is derived. 
Numerical experiments are finally reported to confirm the theoretical result.
\end{abstract}

\section{Introduction}
\label{Intro}
\numberwithin{equation}{section}

Stochastic partial differential equations (SPDEs) appears in many fields of science and engineering, and have been subject of many theoretical and numerical investigations. It is commonly believed that incorporating noise and/or uncertainty into models is closer to reality in mathematical modeling, due to the existence of uncertainty stemming from various sources such as thermal fluctuation, impurities of materials and so on. 
As an active area of research, numerical study of stochastic evolution equations (SEEs) has attracted increasing attention in the past decades; see, e.g., monographs \cite{KLOEDEN2013,LORD2014,MILSTEIN2013,kruse2014strong,hutzenthaler2015numerical,ZhANG2017} and references  therein. 
Although much progress has been made, it is still far from being satisfactory due to the 
numerical approximations to SEEs encounter all the difficulties that may arise in solving deterministic differential equations on one hand,  
and caused by the infinite dimensional nature of the driving noise processes on the other hand.  The present work focus on the SEEs perturbed by a smooth random diffusion coefficient field as well as multiplicative force noise, and aims to propose and analyze an efficient numerical method for this equation.

When considering the numerical approaches for SEEs with various noises, two categories of  convergence errors may be involved, namely weak error and strong error. The former is related to the approximation of the probability law of the solution. Concerning weak convergence error of numerical methods for SEEs, we refer to, for instance,   \cite{printems2001discretization,hausenblas2003weak,de2006weak,debussche2009weak,geissert2009rate,debussche2011weak,kovacs2012weak,kovacs2013weak,lindner2013weak,wang2013weak,brehier2014approximation,andersson2016weak,cui2019strong,cai2021weak} and references therein for a list of literature in this direction. Unlike weak convergence error, the strong convergence error measures the deviation from the trajectory of an exact solution. It has been extensively investigated in various types of SPDEs, see, e.g., \cite{ALLEN1998,davie2001convergence,DU2002,yan2005galerkin,walsh2005finite,JENTZEN2008,gyongy2009rate,KLOEDEN2011,CAO2017,WANGXJ2017,WangXJ2018,Kovcs2018,hutzenthaler2011strong,KRUSE2014,sauer2015lattice,gyongy2016convergence,Feng2017,beccari2019strong,jentzen2020strong,liu2020strong,liu2021strong} and references therein. We mention here some works on strong convergence of the numerical schemes for linear SEEs with additive or multiplicative noise. For example, Allen et al. \cite{ALLEN1998} described, analyzed and compared the finite element and difference methods for parabolic SPDEs driven by additive white noise. Du et al. \cite{DU2002} investigated numerical solutions of linear SEEs perturbed by special additive noises, ranging from the space time white noise to colored noises generated by some infinite dimensional Brownian motions with a prescribed covariance operator. Yan \cite{yan2005galerkin} studied the finite element method for linear SEEs with multiplicative noise in multidimensional case. The case of strong convergence of nonlinear SEEs is generally more subtle and challenging, and has received widely attention in the research community in recent years. For instance, Kloeden et al. \cite{JENTZEN2008,KLOEDEN2011} proposed a discretization based on the Galerkin method in space and exponential integrator in time for the nonlinear SEEs with cylindrical additive noise. Kruse \cite{KRUSE2014} analysed the strong convergence error for a finite element method/linear implicit Euler spatio-temporal discretization of semilinear SEEs with multiplicative noise and Lipschitz continuous nonlinearities, and deduced the optimal error estimates. Wang \cite{WANGXJ2017} derived strong convergence results for a spatio-temporal discretization of the semilinear SEEs with additive noise, where the approximation in space was performed by a standard finite element method and in time by a linear implicit Euler method. Moreover it was shown how exactly the strong convergence rate of the full discretization relies on the regularity of the driven process. Kov$\acute{\mathrm{a}}$cs et al. \cite{Kovcs2018} used Euler type splitstep method to study the semidiscretisation in time of the stochastic Allen-Cahn equation perturbed by smooth additive Gaussian noise, and showed that the strong convergence rate is 1/2 with respect to the step size. Liu et al. \cite{liu2021strong}  proposed a general theory of optimal strong error estimation for some  drift-implicit Euler schemes of a second-order nonlinear SPDE with monotone drift driven by a multiplicative infinite-dimensional Wiener process. 

In this paper, we consider the SEEs with both multiplicative force noise and random diffusion coefficient field, which has not yet been addressed in the literature to the best of our knowledge. The main contributions/novelties of this paper are as follows:

$\bullet$ The well-posedness of the considered stochastic equation is established. 
That is, the existence, uniqueness, and stability of the mild solution is proved.

$\bullet$ The diffusion coefficient considered in the current work is a log-Whittle-Mat\'ern Gaussian random field with
a parametrized covariance function whose regularity can be controlled by a parameter.
Therefore different cases can be tested and compared in a convenient way. 

$\bullet$ A sampling approach called approximation circulant embedding with padding \cite{DIETRICH1997,WOOD1994,NEWSAM1994} is employed to render the equation solvable.  
Then for each sample diffusion coefficient, a time-stepping scheme based on a semi-implicit Euler-Maruyama approach is constructed for the resulting equation. The standard piecewise linear finite element method is employed for the spatial discretization. 
The main theoretical result is the proof of the strong convergence 
rate $\mathcal{O}(h^{2-\varepsilon_0}+\Delta t^{\frac{1}{2}})$ of the full discretization
under certain assumptions, where $\varepsilon_0$ is an infinitesimal positive number, $h$ and $\Delta t$ are respectively the spatial and temporal mesh sizes.

The paper is organised as follows. In Section \ref{sec2}, we establish the well-posedness of the considered problem under given assumptions. The sampling method for the random diffusion coefficient field as well as the spatio-temporal full discretization are presented in Section \ref{sec3}. We devote to deriving the strong error estimate of the proposed fully discrete scheme by using semigroup approach and the stochastic calculus tools in Section \ref{sec5},  and validate the theoretical results by numerical experiments in Section \ref{numericalresult}.

\section{Problem and its well-posedness}\label{sec2}

We start by defining our problem. Let 
$T>0$, 
$D:=(0,1)$, $L^2(D)$ and $H_0^{\gamma}(D)$ are classical Sobolev 
spaces, $\gamma\ge0$. $\mathcal{L}(L^2(D))$ represents the space of bounded linear operators $A$: $L^2(D)\to L^2(D)$ equipped with operator norm $\rVert A\rVert_{_{\L(L^2(D))}}=\sup\limits_{u\ne0}\frac{\rVert Au\rVert_{_{L^2(D)}}}{\rVert u\rVert_{_{L^2(D)}}}$.  
$(\Omega,\mathcal{F},\mathcal{F}_t,\mathbb{P})$ is a filtered probability space 
with a normal filtration $\{\mathcal{F}_t:t\ge0\}$.
Additionally, we denote by $L^2(\Omega,L^2(D))$ the space of all random variables $X: \Omega\to L^2(D)$, such that
$$\| X(\omega) \|_{_{L^2(\Omega,L^2(D))}}<+\infty,\ \ \forall\omega\in\Omega,$$
where the norm $\| \cdot\|_{_{L^2(\Omega,L^2(D))}}$ is defined by
\be\label{L2Omega}
\rVert X(\omega) \rVert_{_{L^2(\Omega,L^2(D))}}:=\mathbb{E}[\rVert X(\omega) \rVert_{_{L^2(D)}}^2]^{\frac{1}{2}}
\ee
with $\mathbb{E}[\cdot]$ being the expectation in the probability space $(\Omega,\mathcal{F},\mathbb{P})$.
$L^2(\Omega,L^2(D))$ is also known as the space of the mean-square integrable random variables.
Let
$W(t,x)$ be a $\mathcal{F}_t$-adapted $H_0^{\gamma}(D)$-valued Wiener process with covariance operator $Q$, where 
$Q$ is a positive definite and symmetric operator with orthonormal eigenfunctions
\{$\phi_j(x)\in H_0^{\gamma}(D):j\in\mathbb{N}$\} and corresponding positive eigenvalues $\{q_j\}$; 
see, e.g., \cite{yan2005galerkin,KRUSE2014,WANGXJ2017} for more details.

Let $Q^{\frac{1}{2}}(H_0^{\gamma}(D)):=\{Q^{\frac{1}{2}}v: v\in H_0^{\gamma}(D)\}$. 
Let $\L_Q$ be the set of linear operators 
$B: Q^{\frac{1}{2}}(H_0^{\gamma}(D))\to L^2(D)$, which satisfies 
\bex
\Big(\sum_{j=1}^\infty \rVert BQ^\frac{1}{2}\phi_j \rVert_{_{L^2(D)}}^2\Big)^\frac{1}{2}<+\infty.
\eex
$\L_Q$ endowed with the norm 
$\rVert B \rVert_{\L_Q}:=\Big(\sum_{j=1}^\infty \rVert BQ^\frac{1}{2}\phi_j \rVert_{_{L^2(D)}}^2\Big)^\frac{1}{2}$ 
is actually the space of Hilbert-Schmidt operators \cite{gohberg1990hilbert}. We will also use the space $L^2(\Omega,\L_Q)$ of all random Hilbert-Schmidt operators 
$B: \Omega\to \L_Q$, equipped with the norm
\bex
\rVert B(\omega) \rVert_{_{L^2(\Omega,\L_Q)}}:=\mathbb{E}[\rVert B(\omega) \rVert_{_{\L_Q}}^2]^{\frac{1}{2}}.
\eex
Throughout the paper we use $c$, with or without subscripts, to mean generic
positive constants ({\it independent of $\omega$ in particular}), which may not be the same at different occurrences.

Our point of interest is the SEE with random diffusion coefficient and
multiplicative noise, written in the abstract form:
\begin{equation}  
\begin{aligned}\label{for12}
du(x,t)&=(-Lu+f(u))dt+G(u) dW(x,t),\  0<{t}<T,\ x\in D,\\
u(x,t)&=0,\ 0\leq t\leq T,\ x\in\partial{D}, \\
u(x,0)&=u_0(x),\ x\in \bar{D},
\end{aligned}
\end{equation}
where $L$ is the elliptic operator $-\partial_x(a(x,\omega)\partial_x)$ 
with the coefficient $a(x,\omega)$ 
being a log-Gaussian random field with the scale parameter $\mathscr{ \varepsilon}$, i.e., 
\begin{align}\label{form13}
	a(x,\omega) & = \mathscr{ \varepsilon} e^{z(x, \omega)}.
\end{align}
This type of random diffusion coefficient field has received a lot of attention in the study of uncertainty quantification (UQ) problems \cite{babuvska2007stochastic,LORD2014}, 
and appeared in some applications, e.g., geostatistical modelling \cite{Wadsworth14,Kazashi19}.
We consider the random field $z(x,\omega)$ 
in \eqref{form13} 
to be a mean-zero Whittle-Mat\'ern Gaussian random field, which is a stationary random field with the covariance function
\begin{equation}\label{form14}
	c_{q}(x):=\frac{\sqrt{2}\Gamma(q+1/2)}{\Gamma(q)}\int_{0}^{\infty}(\frac{2}{\pi})^{1/2}\cos(\lambda x)\frac{1}{(1+\lambda^2)^{q+1/2}}d\lambda, \ x\in [0,1],\ q>2,
\end{equation}
where $\Gamma(\cdot)$ is the Gamma function.

The theoretical result established in this paper depends on the following assumption on the nonlinear drift term $f(\cdot)$:
	\begin{align}
		\Vert f(v)\Vert_{_{L^2(D)}}&\leq c(1+\Vert v\Vert_{_{L^2(D)}}),\  \forall v\in L^2(D),\label{linarofF}\\
			\Vert f(v_1)-f(v_2)\Vert_{_{L^2(D)}}&\leq c(\Vert v_1-v_2\Vert_{_{L^2(D)}}),\ \ \forall v_1,v_2\in L^2(D).\label{lipofF}
	\end{align}
These assumptions are often used to prove the existence and uniqueness of the solution for SPDEs, see, e.g., \cite{KRUSE2014,LORD2014}. 

We are interested in the mild solution of problem (\ref{for12}) in the It\^o sense \cite{DAPA2014},
defined by
\begin{equation}\label{for17}
	u(t)=S(t)u_0+\int_{0}^{t}S(t-\tau)f(u(\tau))d\tau+\int_{0}^{t} S(t-\tau)G(u(\tau)) dW(\tau),
\end{equation}
where $S(t):=e^{-tL}$ is a semigroup generated by the operator $L$ \cite{Engel1999OneparameterSF}.
The well-posedness of the problem \eqref{for12} thus consists in verifying that 
the integrals in \eqref{for17} are well defined and a function $u$ satisfying the integral equation
\eqref{for17} uniquely exists. 
We first notice that 
the realization of the random field $a(x,\omega)$ given in \eqref{form13} is $2$ times mean-square differentiable due to $q>2$ \cite{LORD2014}. Thus, almost surely ($\mathbb{P}$-a.s.), $a(x,\omega)\in C^1(\bar{D})$ and $0<a_{min}(\omega)\leq a(x,\omega)\leq a_{max}(\omega)<\infty$, where $a_{min}(\omega)$ and $a_{max}(\omega)$ represent respectively the essential infimum and supremum of $a(x,\omega)$. 

In order to well define the integral 
$\int_{0}^{t}S(t-s)G(u(s))dW(s)$ and prove the existence and uniqueness of mild solution \eqref{for17}, we assume that there exists
$a_{min}$ and $a_{max}$ such that 
\begin{equation}\label{conditionofa}
	0<a_{min}\leq a_{min}(\omega) \leq a_{max}(\omega)\leq a_{max}<+\infty, \ \ \mathbb{P}\mbox{-a.s.}
\end{equation}
One verifies readily that $\mathcal{D}(L)=H^2(D)\cap H_0^1(D)$ almost surely \cite{BABUSKA2004}, where $\mathcal{D}(L)$ is the domain of the operator $L$.

We also need some assumptions on the nonlinear term $G$, which are collected below:

- $L^{s}G(\cdot)$, $0\leq s\leq\frac{1}{2}$, is a mapping from $L^2(D)$ to $\L_Q$ such that:
\be\label{linearofG}
\Vert L^{s}G(v)\Vert_{_{\L_Q}}\leq c\big(1+\Vert v\Vert_{_{L^2(D)}}\big),\ \ \forall v\in L^2(D),
\ee
\be\label{lipofG}
\big\Vert L^s\big(G(v_1)-G(v_2)\big)\big\Vert_{_{\L_Q}}
\leq c\Vert v_1-v_2\Vert_{_{L^2(D)}},\ \ \forall v_1,v_2\in L^2(D).
\ee

- $\{G(v(\tau)):\tau\in[0,T]\}$ is a predictable $\L_Q$-valued process, 
such that  
\be\label{Gp}
\int_{0}^{T} \mathbb{E}[\rVert G(v)\rVert_{_{\L_Q}}^2]\, d\tau<+\infty, \ \forall v\in L^2(D). 
\ee
\begin{remark}
The assumptions \eqref{linearofG} and \eqref{lipofG} impose some restrictive conditions on the nonlinear term $G(\cdot)$, which include a combination of the nonlinear term $G(\cdot)$, the elliptic operator $L$, and the covariance operator $Q$. Notice that the similar or more general assumptions have been considered in \cite{HAUSENBLAS2003,yan2005galerkin,andersson2016weak}.
\end{remark}

We define the space $\mathbb{L}_2^t$ for $t\in [0,T]$, which is
the Banach space of $L^2(D)$-valued predictable processes $\{v(\tau):\tau\in[0,t]\}$,
equipped with the norm 
\bex
\Vert v\Vert_{\mathbb{L}_2^t}:=\sup\limits_{\tau\in [0,t]}\Vert v(\tau)\Vert_{_{L^2(\Omega,L^2(D))}}<+\infty. 
\eex
Now we are in a position to state and prove the well-posedness of the mild solution to \eqref{for12}.

\begin{theorem}\label{the22}
	Suppose that the initial data $u_0\in L^2(\Omega,L^2(D))$.
	Then, there exists a unique mild solution $u\in \mathbb{L}_2^T$ to \eqref{for12}. 
	Furthermore, the following stability inequality holds
	\begin{equation}\label{theresult}
		\sup_{t\in [0,T]}\Vert u(t)\Vert_{_{L^2(\Omega,L^2(D))}}\leq c_{_T}(1+\Vert u_0\Vert_{_{L^2(\Omega,L^2(D))}}).
	\end{equation} 
\end{theorem}
\begin{proof}
	We define the integral operator $\mathcal{M}$ by: for all $v\in \mathbb{L}_2^t$, $0\leq t\leq T$,
	\begin{equation}\label{f310}
		(\mathcal{M} v)(t):=S(t)u_0+\int_{0}^{t}S(t-\tau)f(v(\tau))d\tau+\int_{0}^{t}S(t-\tau)G(v(\tau)) dW(\tau).
	\end{equation}
	Obviously if there is a fixed point $u\in \mathbb{L}_2^T$ for the operator $\mathcal{M}$, then 
	this fixed point is a mild solution defined by \eqref{for17}. 
	The proof basically consists of two steps: 1) prove that the integral operator $\mathcal{M}$ is well-defined under 
	the assumptions given above; 2) use the Fixed Point Theorem \cite[Theorem 1.10]{LORD2014} to establish the
	existence of a unique mild solution.
	This can be done by following the same lines as 
	in \cite[Theorem 10.26]{LORD2014}, using 
	the imposed assumptions and a number of known results including 
	the Karhunen-Lo{\`e}ve (K-L) expansion of $Q$-Wiener process $W(s)$, It$\hat{\mathrm{o}}$ isometry, and the inequality 
	\be\label{SGS}
	\Vert S(\tau)\Vert_{_{\L(L^2(D))}}\leq 1, \ \ \forall \tau\in (0,T)
	\ee
	for the semigroup $S(\tau)$. We emphasize here that $S(\tau)$ involves the random diffusion coefficient, thus 
	the inequality \eqref{SGS} must be understood in the sense of almost surely. 
	This, compared to the case of deterministic diffusion coefficient
	(see, e.g., 
	\cite[Theorem 10.26]{LORD2014} for details), causes no essential difficulty in establishing the desired results.
\end{proof}

\section{Random field sampling and fully discrete scheme}\label{sec3}
Our first goal in this section is to employ a method called approximation circulant embedding with padding to uniformly sample the random diffusion coefficient $a(x,\omega)$. 
It is notable that some other sampling methods are available, 
such as turning bands method \cite{MANTOGLOU1982,DIETRICH1995} and quadrature sampling method \cite{SHINOZUKA1971,SHINOZUKA1972}. 
However the turning bands method is only applicable to isotropic Gaussian random fields, 
and the quadrature sampling method needs to know the spectral density function of the covariance function of random fields. 
One of the advantages of the sampling method we employ here is its applicability to stationary Gaussian random fields 
including isotropic random fields, and does not require prior knowledge of the spectral density function of the covariance function.

It is obvious from \eqref{form13} that if we want to sample $a(x,\omega)$, we only need to sample $z(x,\omega)$. 
The crucial ingredient of the circulant embedding sampling is that the target covariance matrix can be embedded into a large circulant matrix, which can be decomposed by discrete Fourier transform. Then a new random field based on the combination of decomposition factors is constructed, which will be used to obtain the approximations of $z(x,\omega)$ for $x\in\bar{D}$.

Another purpose in the section is to present semi-implicit Euler-Maruyama scheme and finite element method to discrete problem \eqref{for12} in time and space, respectively. We start by random field sampling.

\subsection{Approximation circulant embedding with padding} 
Consider uniform sampling of random field $z(x,\omega)$ in $\bar{D}:=[0,1]$. We set 
$$ 0= x_1\leq...\leq x_{P}=1, \ \ \Delta x=\frac{1}{P-1}=x_{p+1}-x_{p},\ \ p=1,...,P-1.$$
Let $C:=(c_{ij})$ denote the $P\times P$ covariance matrix with respect to $z(x_p,\omega)$ for $p=1,...,P$, where $c_{ij}:={\rm{cov}}(z(x_{i},\omega),z(x_{j},\omega))=c_{q}(|x_{i}-x_{j}|)$ for $i, j=1,...,P$. If we set $c_{i-j}:=c_{ij}$, then 
\begin{equation}\label{Toeplitz}     
C=\left(                 
\begin{array}{cccc}   
c_0 & c_{-1} & \cdots&c_{1-P}\\ 
c_{1} & c_{0} & \cdots&c_{2-P}\\
\vdots& \ddots &\ddots& \vdots\\
c_{P-1}&\cdots&c_{1}&c_{0}                          
\end{array}
\right).              
\end{equation}
One verifies readily that ${C}$ is a symmetric Toeplitz matrix, and it can be well defined by its first column $\boldsymbol{c_1}=(c_0,...,c_{P-1})^T\in\mathbb{R}^P$. If we define 
$\boldsymbol{\bar{c}_1}:=\tbinom{\boldsymbol{c_1}}{\boldsymbol{0}}\in\mathbb{R}^{P+M}$ with $\boldsymbol{0}\in\mathbb{R}^{M}$ be a zero padding vector, a new symmetric Toeplitz matrix denoted by $\bar{C}\in \mathbb{R}^{(P+M)\times(P+M)}$ can be generated from $\boldsymbol{\bar{c}_1}$. Next, we carry out the minimal circulant extension \cite[Definition 6.48]{LORD2014} to $\bar{C}$ such that it can be embeded into a bigger circulant matrix denoted by $\tilde{\bar{{C}}}\in\mathbb{R}^{2\tilde{P}\times2\tilde{P}}$ for $\tilde{P}:=P+M-1$. Let  $\tilde{\boldsymbol{\bar{c}}}_1$ be the first column of $\tilde{\bar{{C}}}$, $W^{*}$ represent the conjugate transpose of discrete Fourier matrix $W\in\mathbb{C}^{2\tilde{P}\times2\tilde{P}}$, and ${{d}}_j$ be the $j$-th entry of $\sqrt{2\tilde{P}}{{W}}^{*}\tilde{\boldsymbol{\bar{c}}}_1$. Then by Fourier representation, the circulant matrix $\tilde{\bar{{C}}}$ can be decomposed as follows:
 \begin{equation*}\label{dFtwithpadding}
	\tilde{\bar{{C}}}=W(\Lambda_{+}-\Lambda_{-})W^{*},
\end{equation*}
where ${{\Lambda}}_{\pm}$ represents the diagonal matrix whose $j$-th diagonal element is $\pm\lambda_j:=\max\{0,\pm {{d}}_j\}$, i.e.,
\be
{\Lambda}_{\pm}={\rm{diag}}(\pm\lambda_1,\dots,\pm\lambda_{2\tilde{P}}).\label{poslambda}
\ee

Let $\boldsymbol{z}:=\big(z(x_1,\omega),\dots,z(x_{P},\omega)\big)^T$. 
Our main goal is to take the sample approximations to the random vector $\boldsymbol{z}$. To this end, we construct a new random field vector $\boldsymbol{{Z}}$, defined by
\begin{equation}\label{newrandom}
	\boldsymbol{{Z}}:={W}{\Lambda}_{+}^{\frac{1}{2}}\boldsymbol{\xi},\ \  \boldsymbol{\xi}\sim{\rm{CN}}(\boldsymbol{0},2I_{2\tilde{P}}),
\end{equation}
where ${\rm{CN}}(\cdot,\cdot)$ denotes the complex Gaussian distribution \cite[Definition 6.15]{LORD2014}. It's readily to deduce that $\boldsymbol{{Z}}\sim{\rm{CN}}(\boldsymbol{0},2(\tilde{\bar{{C}}}+\tilde{\bar{{C}}}_{-}))$ with $\tilde{\bar{{C}}}_{-}:={{W}}{{\Lambda}}_{-}{{W}}^{*}$, which means both real and imaginary parts of $\boldsymbol{{Z}}$ obey real Gaussian distribution  N$(\boldsymbol{0},(\tilde{\bar{{C}}}+\tilde{\bar{{C}}}_{-}))$. Notice that $\Vert \tilde{\bar{{C}}}_{-}\Vert_2\leq\rho({{\Lambda}}_{-})$ with  $\rho({{\Lambda}}_{-})$ representing the spectral radius of ${{\Lambda}}_{-}$, and it is known that $\rho({{\Lambda}}_{-})$ can be small enough by increasing the dimension $M$ of zero padding vector \cite{WOOD1994}. Therefore  $\tilde{\bar{{C}}}$ can be approximately treated as a non-negative definite matrix when the dimension $M$ is large enough, which is crucial for obtaining a good approximation of the random vector $\boldsymbol{z}$. Then the sample approximations of the random vector $\boldsymbol{z}$ can be provided by truncating the real or imaginary part of $\boldsymbol{Z}$.

The sampling procedure is summarized as follows:

i) Embed $C$ shown in \eqref{Toeplitz} into the padded circulant matrix $\tilde{\bar{{C}}}\in\mathbb{R}^{2(P+M-1)\times2(P+M-1)}$ with dimension $M$ large enough;

ii) Compute $\Lambda_{+}$ by \eqref{poslambda};

iii) Construct a new random field vector $\boldsymbol{Z}$ by \eqref{newrandom} and take its  real or imaginary part, denoted by $\boldsymbol{Z_1}\in\mathbb{R}^{2(P+M-1)}$;

iv) Truncate the first $P$ terms of $\boldsymbol{Z_1}$ and use it as an approximation to the random vector $\boldsymbol{z}$.

It is worthwhile to point out that the sampling method described above is convenient in the sense that it can simultaneously produce two sets of independent and identically distributed (i.i.d) samples in one sampling. 

For each of the sampling data of the random diffusion coefficient, the problem \eqref{for12}
becomes a SEE with randomness only on the $G$-term. 

\subsection{Spatio-temporal discretization}\label{sec4.2}

In this subsection we propose and analyze a discretization method for the problem \eqref{for12}. 
The proposed method is based on a finite element discretization in space and semi-implicit Euler-Maruyama approach in time.

We first describe the $\mathbb{P}_1$ finite element method for the spatial discretization.
Let $K>0, h=\frac{1}{K+1}, x_0=0, x_k=kh, I_k=[x_{k-1},x_k], k=1,\dots, K+1$. Define the finite element space $V_h$ by
$$V_h:=\lbrace v \in C^0(\bar{D}): v|_{I_{k}} \in \mathbb{P}_1(I_k), \ k=1,...,K+1;\  v(0)=v(1)=0 \rbrace,$$
where $\mathbb{P}_1(I_k)$ denotes the space of the polynomials of degree $\le$ 1 defined in $I_k$. Let $\varphi_i(x)$ be the nodal basis functions satisfying $\varphi_i(x_j)=\delta_{ij}, i, j=0,1,\dots,K+1$.
Then $V_h$=${\rm{span}}\lbrace \varphi_1(x),...,\varphi_K(x) \rbrace$. 
Let $\mathcal{P}_{h}$ be the orthogonal projection from $L^2(D)$ to $V_h$, and $\mathcal{P}^{w}_{J}$ be the projection from $H^{\gamma}_0(D)$ to 
the finite-dimensional space ${\rm{span}}\{\phi_1,\dots,\phi{_J}\}$. 
The spatial semi-discrete scheme of the problem \eqref{for12} reads: find finite element approximation $u_h(t)\in V_h$ such that
\begin{equation}
	\begin{aligned}\label{f52}
du_h(t)=\big(-L_hu_h(t)+\mathcal{P}_hf(u_h(t))\big)dt&+\mathcal{P}_h\big(G(u_h(t))\mathcal{P}^{w}_{J}dW(t)\big),\ \forall\ 0<t\leq T,\\
u_h(0)&=\mathcal{P}_h u_0,
	\end{aligned}
\end{equation} 
where $L_h$: $V_h\to V_h$ is the finite-dimensional operator defined by
\begin{equation*}
(L_hw,v):=(a(x,\omega)\partial_x w, \partial_x v), \ \ \forall w,v\in V_h 
\end{equation*}
with $(\cdot,\cdot)$ be the $L^2$-inner product. 

We now describe the temporal discretization. 
Let $N$ be a positive integer, $\Delta t:=T/N$ be the uniform time step. 
Then the spatio-temporal full discretization of the problem \eqref{for12}, called hereafter the finite element method/semi-implicit Euler Maruyama scheme, reads:
\begin{equation}
	\begin{aligned}\label{f53}
(I+\Delta tL_h)u^{n+1}_{h}=u^n_{h}&+\Delta t\mathcal{P}_hf(u^n_{h})+\mathcal{P}_h\big(G(u^n_{h})\mathcal{P}^{w}_{J}\Delta W^n\big),\ n=0,...,N-1,\\
 u^0_{h}&=\mathcal{P}_hu_0,
	\end{aligned}
\end{equation}
where $\mathcal{P}^{w}_{J}\Delta W^n:=\sum_{j=1}^{J}  \sqrt{q_j}(\beta_j(t_{n+1})-\beta_j(t_{n})) \phi_j$ with $\beta_j(t)$ be the i.i.d $\mathcal{F}_t$-Brownian motions. 

Before carrying out the error analysis, we briefly discuss the implementation of the above scheme. 
The weak formulation of \eqref{f53} is:
\be\label{variationalf}
(u_h^{n+1},v_h)+\Dt(a(x,\omega)\partial_x u_h^{n+1},\partial_x v_h)=(g_h^n,v_h),\ \ v_h\in V_h,
\ee
where $g_h^n:=u_h^n+\Dt f(u_h^n)+G(u_h^n)\mathcal{P}_J^w\Delta W^n$. 
Expressing the solution $u_h^{n+1}$ under the basis $\{\varphi_k\}_{k=1}^K$, 
\bex
u^{n+1}_{h}(x)=\sum_{k=1}^{K}\hat{u}^{n+1}_{k}\varphi_k(x),\ n=0,...,N-1,
\eex
and taking the test function $v_h$ in \eqref{variationalf} to be each of the basis functions,
we arrive at the following linear system:
\begin{equation*}
	(M+\Delta tS)\hat{\boldsymbol{u}}^{n+1}_{h}=M\hat{\boldsymbol{g}}^n_{h},\ n=0,\dots,N-1,
\end{equation*} 
where $\hat{\boldsymbol{u}}^{n+1}_{h}:=(\hat{u}^{n+1}_{1},\dots,\hat{u}^{n+1}_{K})^T$, $\hat{\boldsymbol{g}}^n_{h}$ is the expansion coefficient vector of $g_h^n$ under the basis $\{\varphi_k\}_{k=1}^K$. $M$ and $S$ are respectively the mass and stiffness matrix defined by
\begin{equation*}
	\begin{aligned}
	M&=(m_{ij}),\ \ m_{ij}:=(\varphi_i,\varphi_j ),\ \ \forall i,j=1,...K,\\
	S&=(s_{ij}),\ \  s_{ij}:=(a(x,\omega)\partial_x\varphi_i,\partial_x\varphi_j), \ \ \forall i,j=1,...,K.
	\end{aligned}
\end{equation*}
In actual calculation,  we will use $\frac{1}{2}\big(a(x_{k-1},\omega)+a(x_{k},\omega)\big)$ to approximate $a(x,\omega)$ for $x\in I_k$.
Therefore the overall cost of the scheme is roughly equal to solving a linear system with random variable coefficients at each time step.

\section{Error estimate}\label{sec5}
This section is devoted to analyzing the strong convergence error of the spatio-temporal full discretization \eqref{f53} to the mild solution \eqref{for17}. Here, strong convergence is understood in the sense of convergence with respect to the norm $\|\cdot\|_{L^2(\Omega,L^2(D))}$. We first note that the full-discrete scheme \eqref{f53} can be rewritten 
under form:
\begin{equation}\label{invfulldis}
	u^{n+1}_{h}=(I+\Delta tL_h)^{-1}\Big(u^n_{h}
	+\Delta t\mathcal{P}_{h}f(u^n_{h})
	+\mathcal{P}_{h}G(u^n_{h})\mathcal{P}^w_{J}\Delta W^n\Big),\ n=0,\dots,N-1.
\end{equation} 
It is readily seen that $L_h$ is reversible in $V_h$, 
i.e., $L_h^{-1}v_h$ is well defined for all $v_h\in V_h$. We now extend the definition of $L_h^{-1}$ 
to all $v\in L^2(D)$ by $L_h^{-1}v = L_h^{-1}\mathcal{P}_{h}v$.
By the assumption on $a(x,\omega)$, we know that for almost every $\omega\in\Omega$, 
$L_h^{-1}$ is a non-negative definite operator from $L^2(D)$ to $V_h$. In fact, for all $v\in L^2(D)$, there
exists $w_h\in V_h$ such that $L_h w_h = \mathcal{P}_{h}v$, and thus 
\begin{align}
	(L_h^{-1}v, v)
	&=(L_h^{-1}\mathcal{P}_{h}v, v) = (L_h^{-1}\mathcal{P}_{h}v, \mathcal{P}_{h}v)=(L_h^{-1}L_h w_h, L_h w_h)\notag\\
	& = (w_h, L_h w_h) = \big(a(x,\omega)\partial_xw_h, \partial_xw_h\big) \ge 0\notag.
\end{align}

Let $S_{h,\Delta t}^{n}:=(I+\Delta tL_h)^{-n}$. 
The fully discrete approximation can be expressed under the form:
\be
u^n_{h}=S_{h,\Delta t}^{n}\mathcal{P}_{h}u_0+\sum_{k=0}^{n-1}\Delta tS_{h,\Delta t}^{n-k}\mathcal{P}_{h}f(u^k_{h})+\sum_{k=0}^{n-1}\int_{t_k}^{t_{k+1}}S_{h,\Delta t}^{n-k}\mathcal{P}_{h}G(u^k_{h})\mathcal{P}^w_{J}dW(\tau).
\label{recur}
\ee
Subtracting \eqref{recur} from the mild solution \eqref{for17} gives 
\be\label{errdec}
u(t_n)-u^n_{h}=\theta_1+\theta_2+\theta_3
\ee
with $\theta_i, i=1,2,3$, representing
\begin{align}
	&\theta_1:=S(t_n)u_0-S_{h,\Delta t}^n\mathcal{P}_{h}u_0,\label{errore1}\\
	&\theta_2:=\sum_{k=0}^{n-1}\big(\int_{t_k}^{t_{k+1}}S(t_n-\tau)f(u(\tau))d\tau-\Delta tS_{h,\Delta t}^{n-k}\mathcal{P}_{h}f(u^k_{h})\big),\label{errore2}\\
	&\theta_3:=\sum_{k=0}^{n-1}\int_{t_k}^{t_{k+1}}\big(S(t_n-\tau)G(u(\tau))-S_{h,\Delta t}^{n-k}\mathcal{P}_{h}G(u^k_{h})\mathcal{P}^w_{J}\big)dW(\tau).\label{errore3}
\end{align}

Our goal in the following is to estimate $\theta_1$, $\theta_2$, $\theta_3$ separately in the sense of strong convergence. To this end, we first give some preliminaries that will be used in subsequent analysis.

$\bullet$ If the initial value $u_0\in L^2(\Omega,\mathcal{D}(L))$, then there exists a constant $c$ depended on $u_0$ such that the mild solution $u$ defined in \eqref{for17} satisfies the following temporal H{\"o}lder regularity: 
\be
\Vert u(\tau_2)-u(\tau_1)\Vert_{_{L^2(\Omega,L^2(D))}}\leq c(\tau_2-\tau_1)^{\frac{1}{2}},\ \ \forall\ 0\leq \tau_1\leq \tau_2\leq T\label{timereg}.
\ee
The proof of \eqref{timereg} can be done by following the same lines as in \cite[Lemma 10.27]{LORD2014}, which is omitted here. Basically, it makes use of the properties of the operator $L$ and its associated semigroup $S(t)$, satisfied in the sense of almost surely. 

$\bullet$ The operator $L$ and the induced semigroup $S(t)$ satisfy the following estimates, 
which is a straightforward extension of the classical results (see, e.g., \cite{KRUSE2014,WANGXJ2017}) to the sense of almost surely:

- For each $\alpha\ge 0$, there exists a constant $c$ such that
\begin{align}\label{posalpha}
	\| L^{\alpha}S(t)\|_{_{\L(L^2(D))}}\leq c t^{-\alpha},\ \forall t>0.
\end{align}

- For $\alpha\in[0,1]$, there exists a constant $c$ such that
\begin{align}\label{negalpha}
	\| L^{-\alpha}(I-S(t))\|_{_{\L(L^2(D))}}\leq ct^{\alpha},\ \forall t\ge 0.
\end{align}

$\bullet$ The nonlinear term $G$ satisfies
\be\label{f617}
\| G(v_1)\mathcal{P}^w_{J}-G(v_2)\mathcal{P}^w_{J}\|_{_{\L_Q}}\leq c\| v_1-v_2\|_{_{L^2(D)}},\ \ \forall v_1,v_2\in{L^2(D)}.
\ee
In fact, it follows from the assumption \eqref{lipofG}: for all
$v_1,v_2\in{L^2(D)}$, 
\begin{equation*}
	\begin{aligned}
		\| G(v_1)\mathcal{P}^w_{J}-G(v_2)\mathcal{P}^w_{J}\|_{_{\L_Q}}
		&=\Big[\sum_{j=1}^{\infty}\big\| \big(G(v_1)-G(v_2)\big)\mathcal{P}^w_{J}Q^{1/2}\phi_j\big\|_{_{L^2(D)}}^2\Big]^\frac{1}{2}\\
		&=\Big[\sum_{j=1}^{J}\big\| \big(G(v_1)-G(v_2)\big)Q^{1/2}\phi_j\big\|_{_{L^2(D)}}^2\Big]^\frac{1}{2}\\
		&\leq \big\| G(v_1)-G(v_2)\big\|_{_{\L_Q}}\leq c\| v_1-v_2\|_{_{L^2(D)}}.
	\end{aligned}
\end{equation*}

\subsection{Strong error estimate}
We first focus on strong error estimate for the term $\theta_1$. 
It is worth pointing out that, although our analysis is inspired by the work
\cite{KRUSE2014} based on the rational function approach, our proof makes full use of the standard framework of the finite element approximation to the linear parabolic equation as well as the fact that operator $L_h^{-1}$ is non-negative definite from $L^2(D)$ to $V_h$. 
Let
	\be\label{Tn}
	T_n:=\big(e^{-t_nL}-(I+\Delta tL_h)^{-n}\mathcal{P}_{h}\big).
	\ee
Then it follows from the definition \eqref{errore1} that $\theta_1=T_nu_0$.
\begin{lemma}[Error estimate of $\theta_1$]\label{lem63}
	Suppose $u_0\in L^2(\Omega,\mathcal{D}(L))$. 
	Then there exists a constant $c$ independent of $h$ and $\Delta t$ (but depends on $u_0$), 
	such that
	\bex
	\| \theta_1\|_{_{L^2(\Omega,L^2(D))}}\leq c(\Delta t+h^2),
	\eex
	where $\theta_1$ is given in \eqref{errore1}.
\end{lemma}
\begin{proof}
	Obviously, $\theta_1$ characterizes the error between the exact solution 
	$e^{-t_nL}u_0$ and the full discrete solution 
	$(I+\Delta tL_h)^{-n}\mathcal{P}_{h}u_0$, which can be splited into two parts:
	\be\label{errorsplit}
	\theta_1 = e^n_1 + e^n_2,
	\ee
	where
	\bex
	e^n_1:=e^{-t_nL}u_0-e^{-t_nL_h}\mathcal{P}_{h}u_0
	\eex
	is the spatial discretization error, while
	\bex
	e^n_2:=\big(e^{-t_nL_h}-(I+\Delta tL_h)^{-n}\big)\mathcal{P}_{h}u_0
	\eex
	is the temporal discretization error. 
	Clearly, we have $e^n_1= y(t_n) - y_h(t_n)$, where 
	$y(t)\in H_0^1(D)$ and $y_h(t) \in V_h$  
	are the solutions of the parabolic equation
	\bex
	\frac{\partial y(t)}{\partial t} + Ly(t)=0, \ \ \  y(0)=u_0
	\eex
	and its finite element semi-discrete equation
	\bex
	\frac{\partial y_h(t)}{\partial t} + L_h y_h(t)=0, \ \ \ y_h(0) = \mathcal{P}_{h}u_0
	\eex
	respectively. 
	Let
	\begin{equation*}\label{definitionrhoande}
		e_1(t):=y(t)-{y}_h(t),\ \ \ \ \rho(t):=L_h^{-1}\frac{\partial e_1(t)}{\partial t}+e_1(t).
	\end{equation*}
	It can be verified that
	\bex
	\rho(t)=(L^{-1} - L_h^{-1}) Ly(t).
	\eex
    Using the non-negative definite of the operator $L_h^{-1}$ as well as the standard error analysis of the finite element approximation to the parabolic equation \cite[Lemma 3.51]{LORD2014} gives: 
	for almost every $\omega\in\Omega$, 
	\begin{equation*}\label{important1}
		\| e_1(t)\|_{_{L^2(D)}}
		\leq c\sup_{0\le \tau\le t}\Big(\|\rho(\tau)\|_{_{L^2(D)}}+\tau\Big\|\frac{\partial \rho(\tau)}{\partial \tau}\Big\|_{{L^2(D)}}\Big).
	\end{equation*}
The terms in the right-hand side can be bounded by:
	\begin{equation*}\label{important2}
		\|\rho(\tau)\|_{_{L^2(D)}}=\|(L^{-1}-L_h^{-1})Ly(\tau)\|_{_{L^2(D)}}\leq ch^2\| Le^{-L\tau}u_0\|_{_{L^2(D)}}\leq ch^2,
	\end{equation*}
	\begin{equation*}\label{important3}
		\tau\Big\|\frac{\partial\rho(\tau)}{\partial \tau}\Big\|_{{L^2(D)}}=\tau\| (L^{-1}-L_h^{-1})L^2y(\tau)\|_{_{L^2(D)}}
		\leq c\tau h^2\| L^2e^{-L\tau}u_0\|_{_{L^2(D)}}
		\leq ch^2\|Lu_0\|_{_{L^2(D)}}\leq ch^2.
	\end{equation*}
	Thus
	\begin{equation}\label{for4.22}
		\|e_1^n\|_{_{L^2(D)}} = \| e_1(t_n)\|_{_{L^2(D)}}\leq ch^2,\ \ \ n=0,1,\dots, N.
	\end{equation}
	We now turn to estimate the temporal discretization error $e^n_2$. A direct calculation gives
	\begin{equation*}
		\begin{aligned}
			\| e^n_2 \|_{_{L^2(D)}}&=\big\|\big(e^{-n\Delta tL_h}-(I+\Delta tL_h)^{-n}\big)L_h^{-1}L_h
			\mathcal{P}_{h}u_0\big\|_{{L^2(D)}}\\
			&\leq\big\|\big(e^{-n\Delta tL_h}-(I+\Delta tL_h)^{-n}\big)L_h^{-1}\big\|_{{\mathcal{L}({L^2(D)})}}
			\| L_hP_hu_0\|_{_{L^2(D)}}.
		\end{aligned}
	\end{equation*}
	Noticing that the operator $L_h$ is symmetric, and 
	the $\L(L^2(D))$-norm of the operator $\big(e^{-n\Delta tL_h}-(I+\Delta tL_h)^{-n}\big)L_h^{-1}$ is equal to its spectral radius, i.e.,
	\bex
	\sup_{j=1,...,K}\big|\big(e^{-n\Delta t\lambda^h_{j}}-(1+\Delta t\lambda^h_{j})^{-n}\big)
	/\lambda^h_{j}\big|, 
	\eex
	where $\lambda^h_{j}>0$, $j=1,...,K$, are the eigenvalues of $L_h$. 
	Note that $|(e^{-nx}-(1+x)^{-n})/x|$ is bounded for $x>0$, therefore taking 
	$x=\lambda_{j}^h\Delta t$ gives
	\bex
	&&\sup_{j=1,...,K}\big|\big(e^{-n\Delta t\lambda^h_{j}}-(1+\Delta t\lambda^h_{j})^{-n}\big)
	/\lambda^h_{j}\big| \le c\Dt. 
	\eex
	This proves
	\begin{equation}\label{2bound}
		\| e^n_2\|_{_{L^2(D)}}\leq c\Delta t,\ \  n=0,1,\dots,N.
	\end{equation}
	Combining \eqref{errorsplit}, \eqref{for4.22}, and \eqref{2bound} gives
	$$\| \theta_1\|_{_{L^2(D)}}\leq c(\Delta t+h^2).$$
	The above estimate holds for almost all $\omega\in \Omega$. Therefore
	$$\|\theta_1\|_{_{L^2(\Omega,L^2(D))}}\leq c(\Delta t+h^2).$$ 
\end{proof} 
\begin{remark} If $u_0\in L^2(\Omega,L^2(D))$. Then we have only \cite{KRUSE2014,LORD2014}: 
for almost every $\omega\in\Omega$, 
	\be\label{operatorbound1}
	\|\theta_1\|_{_{L^2(D)}} = \|T_n u_0\|_{_{L^2(D)}} \leq c\|u_0\|_{_{L^2(D)}} \frac{\Delta t+h^2}{t_n},\ \ n=1,\dots,N.
	\ee
\end{remark}

We next derive the error estimate for the term $\theta_2$, which is based on the standard error analysis 
for the deterministic semilinear evolution equation, the semigroup property, 
and the temporal regularity of the mild solution. 

\begin{lemma}[Error estimate of $\theta_2$]\label{lem64}
	Suppose $u_0\in L^2(\Omega,\mathcal{D}(L))$. Then 
there exists 
a constant $c$ independent of $h$ and $\Dt$ (but depends on $\|u_0\|_{L^2(\Omega,\mathcal{D}(L))}$), such that
	\begin{equation*}\label{lemma5.4re}
		\|\theta_2\|_{_{L^2(\Omega,L^2(D))}}\leq c\Big[\Delta t^{\frac{1}{2}}+(\Delta t+h^2)\ln(\Delta t^{-1})+\sum_{k=0}^{n-1}\| u(t_k)-u^k_{h}\|_{_{L^2(\Omega,L^2(D))}}\Delta t\Big], \ \ \ n=1,\dots,N, 
	\end{equation*}
where $\theta_2$ is given by \eqref{errore2}.
\end{lemma}
\begin{proof}
	The term to be bounded can be expressed by
	\bex
	\theta_2=\sum_{k=0}^{n-1}\int_{t_k}^{t_{k+1}}\big(S(t_n-\tau)f(u(\tau))-S_{h,\Delta t}^{n-k}\mathcal{P}_{h}f(u^k_{h})\big)d\tau,
	\eex
which can be decomposed into
	\bex
	\theta_2= \theta_2^1 + \theta_2^2 + \theta_2^3 + \theta_2^4
	\eex
	with
	\bex
	&& \theta^1_2:=\sum_{k=0}^{n-1}\int_{t_k}^{t_{k+1}} \big(S(t_n-\tau)-S(t_n-t_k)\big)f(u(\tau)) d\tau,\\
	&& \theta^2_2:=\sum_{k=0}^{n-1}\int_{t_k}^{t_{k+1}} \big(S(t_n-t_k)-S_{h,\Delta t}^{n-k}\mathcal{P}_{h}\big)f(u(\tau))d\tau,\\
	&& \theta^3_2:=\sum_{k=0}^{n-1}\int_{t_k}^{t_{k+1}} S_{h,\Delta t}^{n-k}\mathcal{P}_{h}\big(f(u(\tau))-f(u(t_k))\big)d\tau,\\
	&& \theta^4_2:=\sum_{k=0}^{n-1}\int_{t_k}^{t_{k+1}} S_{h,\Delta t}^{n-k}\mathcal{P}_{h}\big(f(u(t_k))-f(u^k_{h})\big)d\tau.
	\eex
For the part $\theta^1_2$, it follows from the norm definition \eqref{L2Omega}:
\begin{equation*}
	\begin{aligned}\label{partI}
		\|\theta^1_2\|_{_{L^2(\Omega,L^2(D))}}
		\leq\sum_{k=0}^{n-1}\int_{t_k}^{t_{k+1}}\mathbb{E}\Big[\| S(t_n-\tau)-S(t_n-t_k)\|_{_{\mathcal{L}({L^2(D)})}}^2\| f(u(\tau))\|_{_{L^2(D)}}^2\Big]^{\frac{1}{2}}d\tau.
	\end{aligned}
\end{equation*}
According to \eqref{posalpha} and \eqref{negalpha}, the operator norm $\| S(t_n-\tau)-S(t_n-t_k)\|_{_{\mathcal{L}({L^2(D)})}}$ is bounded
$\mathbb{P}$-a.s. by:
\bex
& \| S(t_n-\tau)-S(t_n-t_{n-1})\|_{_{\mathcal{L}({L^2(D)})}} &\le c,
\\
& \| S(t_n-\tau)-S(t_n-t_k)\|_{_{\mathcal{L}({L^2(D)})}}
		&\leq\| LS(t_n-\tau)\|_{_{\mathcal{L}({L^2(D)})}}\| L^{-1}(I-S(\tau-t_k))\|_{_{\mathcal{L}({L^2(D)})}}\\
		&&\leq c\frac{\tau-t_k}{t_n-\tau}, \ \ \tau\in (t_k, t_{k+1}),\ \ k=0,\dots,n-2.
\eex
We further use 
\eqref{linarofF} and \eqref{theresult} to derive
\begin{equation*}
	\begin{aligned}
\|\theta^1_2\|_{{L^2(\Omega,L^2(D))}}
 &\leq c\Big(\Delta t+\sum_{k=0}^{n-2}\int_{t_k}^{t_{k+1}}\frac{\Delta t}{t_n-t_{k+1}}d\tau\Big)\leq c\Dt\Big(1+\sum_{k=1}^{n}\frac{1}{k}\Big)\\
&\leq c\Dt(1+\ln n)\leq c\Dt(1+\ln(\Dt^{-1})).
	\end{aligned}
\end{equation*}
The estimate of $\theta^2_2$ follows from \eqref{Tn}, \eqref{operatorbound1}, \eqref{linarofF}, and \eqref{theresult}:
	\begin{equation*}
		\begin{aligned}
			\|\theta^2_2\|_{_{L^2(\Omega,L^2(D))}}&\leq\sum_{k=0}^{n-1}\int_{t_k}^{t_{k+1}}\| T_{n-k}f(u(\tau))\|_{_{L^2(\Omega,L^2(D))}}d\tau\\
			&\leq\sum_{k=0}^{n-1}\int_{t_k}^{t_{k+1}}\mathbb{E}\big[\| T_{n-k}\|_{_{\L{(L^2(D))}}}^2\| f(u(\tau))\|_{_{L^2(D)}}^2\big]^{\frac{1}{2}}d\tau
			\\&\leq c\sum_{k=0}^{n-1}\int_{t_k}^{t_{k+1}}\frac{\Delta t+h^2}{t_{n-k}}d\tau
			= c(\Delta t+h^2)\sum_{k=0}^{n-1}\frac{1}{n-k} \le c(\Delta t+h^2)\ln(\Dt^{-1}).
		\end{aligned}
	\end{equation*}
	For the part $\theta^3_2$, by $\| S_{h,\Delta t}^{n-k}\|_{_{\mathcal{L}({L^2(D)})}}\leq1$ ($\mathbb{P}\mbox{-a.s.}$), $\| \mathcal{P}_{h}\|_{_{\mathcal{L}({L^2(D)})}}\leq1$, \eqref{lipofF}, and \eqref{timereg}, we have
	\begin{equation*}\label{5.97}
		\begin{aligned}
			\|\theta^3_2\|_{_{L^2(\Omega,L^2(D))}}&\leq\sum_{k=0}^{n-1}\int_{t_k}^{t_{k+1}}\big\| f(u(\tau))-f(u(t_k))\big\|_{{L^2(\Omega,L^2(D))}}d\tau
			\\&\leq c\sum_{k=0}^{n-1}\int_{t_k}^{t_{k+1}}\| u(\tau)-u(t_k)\|_{_{L^2(\Omega,L^2(D))}}d\tau
			\\&\leq c\sum_{k=0}^{n-1}\int_{t_k}^{t_{k+1}}(\tau-t_k)^{\frac{1}{2}}d\tau
			\leq c\Delta t^{\frac{1}{2}}.
		\end{aligned}
	\end{equation*}
The part $\theta^4_2$ can be estimated similarly: 
	\begin{equation*}\label{5.98}
		\begin{aligned}
			\| \theta^4_2\|_{_{L^2(\Omega,L^2(D))}}
			&\leq \sum_{k=0}^{n-1}\int_{t_k}^{t_{k+1}}\| f(u(t_k))-f(u^k_{h})\|_{_{L^2(\Omega,L^2(D))}}d\tau
			\\&\leq c\sum_{k=0}^{n-1}\int_{t_k}^{t_{k+1}}\| u(t_k)-u^k_{h}\|_{_{L^2(\Omega,L^2(D))}}d\tau
			\\&= c\sum_{k=0}^{n-1}\| u(t_k)-u^k_{h}\|_{_{L^2(\Omega,L^2(D))}}\Dt.
		\end{aligned}
	\end{equation*}
Finally, we conclude by combining all above estimates with the triangle inequality. 
\end{proof} 
In order to estimate the error contribution term $\theta_3$, we need to derive an estimate related to the nonlinear term $G$.
\begin{lemma}\label{lem61}	
	Suppose $u_0\in L^2(\Omega,\mathcal{D}(L))$, and the eigenvalues of $Q$ satisfy $q_j=\mathcal{O}(j^{-(2\gamma+1+{\epsilon})})$ for some $\gamma\ge0$ and ${\epsilon}>0$. Then it holds: for $0\leq \tau_1\leq \tau_2\leq T$,
\be
	\big\| \mathcal{P}_{h}\big(G(u(\tau_2))-G(u(\tau_1))\mathcal{P}^w_{J}\big)\big\|{_{_{L^2(\Omega,\L_Q)}}}\leq c(|\tau_2-\tau_1|^{\frac{1}{2}}+J^{-\gamma}).\label{f618}
\ee
\end{lemma}
\begin{proof}
 Using the triangle inequality:
	\begin{equation*}
		\begin{aligned}\label{f4.11}
			\big\| \mathcal{P}_{h}\big(G(u(\tau_2))-G(u(\tau_1))\mathcal{P}^w_J\big)\big\|_{_{L^2(\Omega,\L_Q)}}&\leq\big\|\mathcal{P}_{h}\big(G(u(\tau_2))-G(u(\tau_1))\big)\big\|_{_{L^2(\Omega,\L_Q)}}\\
			&\ \ \ +\big\| \mathcal{P}_{h}\big(G(u(\tau_1))-G(u(\tau_1))\mathcal{P}^w_{J}\big)\big\|_{_{L^2(\Omega,\L_Q)}},
		\end{aligned}
	\end{equation*}
	we are led to estimate the two terms in the right-hand side.
	First, employing \eqref{lipofG} and \eqref{timereg} gives:
	\begin{equation*}\label{in6}
		\begin{aligned}
			\big\|\mathcal{P}_{h}\big(G(u(\tau_2))-G(u(\tau_1))\big)\big\|_{_{L^2(\Omega,\L_Q)}}
			&\leq\big\| G(u(\tau_2))-G(u(\tau_1))\big\|_{_{L^2(\Omega,\L_Q)}}\\
			&\leq
			c\| u(\tau_2)-u(\tau_1)\|_{_{L^2(\Omega,L^2(D))}}\leq c|\tau_2-\tau_1|^{\frac{1}{2}}.
		\end{aligned}
	\end{equation*}
	Then under the assumptions \eqref{linearofG} and \eqref{theresult}, we have
	\begin{equation*}
		\begin{aligned}\label{f4.13}
			\big\| \mathcal{P}_{h}\big(G(u(\tau_1))-G(u(\tau_1))\mathcal{P}^w_{J}\big)\big\|_{_{L^2(\Omega,\L_Q)}}
			&\leq\mathbb{E}\Big[\sum_{j=1}^{\infty}\big\| G(u(\tau_1))(I-\mathcal{P}^w_{J})Q^{\frac{1}{2}}\phi_j\big\|_{_{L^2(D)}}^2\Big]^\frac{1}{2}\\
			&\leq c \mathbb{E}\Big[\big\| G(u(\tau_1))\big\|_{_{\L{(L^2(D))}}}^2
			\sum_{j=1}^{\infty}\big\|(I-\mathcal{P}^w_{J})q_j^{\frac{1}{2}}\phi_j\big\|_{_{L^2(D)}}^2\Big]^\frac{1}{2}\\
			&\leq c(1+\| u_0\|_{_{L^2(\Omega,L^2(D))}})
			\Big(\sum_{j=J+1}^{\infty}q_j\Big)^{1/2}\\
			&\le cJ^{-\gamma}.
		\end{aligned}
	\end{equation*}
	This proves \eqref{f618}.
\end{proof}

\begin{lemma}[Error estimate of $\theta_3$]\label{lem65}
Under the assumptions of Lemma \ref{lem61}, further assume $\Dt=O(h^2)=O(J^{-\gamma})$.  
	Then there exists a constant $c$ independent of $\Dt$ and $h$, such that 
	\begin{equation}\label{iii5}
		\| \theta_3\|^2_{_{L^2(\Omega,L^2(D))}}\leq c\Big[(\Delta t^{\frac{1}{2}}+h^2)^2+\sum_{k=0}^{n-1}\| u(t_k)-u^k_{h}\|^2_{_{L^2(\Omega,L^2(D))}}{\Delta t}\Big],\ \ n=1,\dots,N,
	\end{equation}
where $\theta_3$ is given by \eqref{errore3}.
\end{lemma}
\begin{proof}
Split $\theta_3$ as $\theta_3=\sum_{i=1}^{4}\theta_3^i$, where 
\bex
\theta_{3}^i:=\sum_{k=0}^{n-1}\int_{t_k}^{t_{k+1}}{X_i}dW(\tau)
\eex
with
\begin{align}
	&{X_1}:=\big(S(t_n-\tau)-S(t_n-t_k)\big)G(u(\tau)),\quad {X_2}:=\big(S(t_n-t_k)-S_{h,\Delta t}^{n-k}\mathcal{P}_{h}\big)G(u(\tau)),\notag\\
	&{X_3}:=S_{h,\Delta t}^{n-k}\mathcal{P}_{h}\big(G(u(\tau))-G(u(t_k))\mathcal{P}^w_{J}\big),\ {X_4}:=S_{h,\Delta t}^{n-k}\mathcal{P}_{h}\big(G(u(t_k))\mathcal{P}^w_{J}-G(u^k_{h})\mathcal{P}^w_{J}\big)\notag.
\end{align}
For the part $\theta_3^{1}$, it follows from the It\^o isometry:
\begin{equation*}
	\begin{aligned}\label{or1}
		\| \theta_3^{1}\|_{_{L^2(\Omega,L^2(D))}}^2&=\sum_{k=0}^{n-1}\int_{t_k}^{t_{k+1}}\mathbb{E}[\|{X_1}\|_{_{\L_Q}}^2]d\tau\\
		&\leq\int_{t_{n-1}}^{t_n}\mathbb{E}\Big[\big\|S(t_n-\tau)-S(t_n-t_{n-1})\big\|^2_{_{\L(L^2(D))}}\big\| G(u(\tau))\big\|_{_{\L_Q}}^2\Big]d\tau\\
		&\ \ \ +\sum_{k=0}^{n-2}\int_{t_k}^{t_{k+1}}\mathbb{E}\Big[\big\|\big(S(t_n-\tau) -S(t_n-t_k)\big)L^{-\frac{1}{2}}\big\|_{_{\mathcal{L}({L^2(D)})}}^2\| L^{\frac{1}{2}}G(u(\tau))\|_{_{\L_Q}}^2\Big]d\tau.
	\end{aligned}
\end{equation*}
We are led to estimate the two terms on the right-hand side of the inequality. First using $\big\|S(t_n-\tau)-S(t_n-t_{n-1})\big\|^2_{_{\L(L^2(D))}}\leq c$ ($\mathbb{P}$-a.s.), \eqref{linearofG}, and \eqref{theresult} yields
\bex
\int_{t_{n-1}}^{t_n}\mathbb{E}\Big[\big\|S(t_n-\tau)-S(t_n-t_{n-1})\big\|^2_{_{\L(L^2(D))}}\big\| G(u(\tau))\big\|_{_{\L_Q}}^2\Big]d\tau\leq c\Dt.
\eex
Then employing \eqref{posalpha} and \eqref{negalpha} gives
\begin{equation*}
	\begin{aligned}
 \big\|\big(S(t_n-\tau) -S(t_n-t_k)\big)L^{-\frac{1}{2}}\big\|&_{_{\mathcal{L}({L^2(D)})}}^2=\big\| L^{\frac{1}{2}}S(t_n-\tau) L^{-1}\big(I-S(\tau-t_k)\big)\big\|_{_{\mathcal{L}({L^2(D)})}}^2\\
&\leq \big\| L^{\frac{1}{2}}S(t_n-\tau)\big\|_{_{\mathcal{L}({L^2(D)})}}^2\big\| L^{-1}\big(I-S(\tau-t_k)\big)\big\|_{_{\mathcal{L}({L^2(D)})}}^2\\
&\leq c\frac{(\tau-t_k)^{2}}{t_n-\tau}.
\end{aligned}
\end{equation*}
 Making use of \eqref{linearofG}, \eqref{theresult} gives
 \begin{equation*}
 	\begin{aligned}
 \sum_{k=0}^{n-2}\int_{t_k}^{t_{k+1}}\mathbb{E}\Big[\big\|\big(&S(t_n-\tau) -S(t_n-t_k)\big)L^{-\frac{1}{2}}\big\|_{_{\mathcal{L}({L^2(D)})}}^2\| L^{\frac{1}{2}}G(u(\tau))\|_{_{\L_Q}}^2\Big]d\tau\\
 &\leq c\sum_{k=0}^{n-2}\int_{t_k}^{t_{k+1}}\frac{\Delta t^{2}}{t_n-t_{k+1}}d\tau
 \le c \Dt^2\ln(\Dt^{-1}).
 	\end{aligned}
\end{equation*}
Therefore
\begin{equation*}
	\begin{aligned}\label{iii1}
		\| \theta_3^{1}\|_{_{L^2(\Omega,L^2(D))}}^2
	\leq c(\Dt + \Dt^2\ln(\Dt^{-1})).
	\end{aligned}
\end{equation*}
The estimate for $\theta_3^{2}$ follows from It\^o isometry, \eqref{Tn}, and \eqref{operatorbound1}:
\begin{equation*}
\begin{aligned}
\| \theta_3^{2}\|_{_{L^2(\Omega,L^2(D))}}^2
=\sum_{k=0}^{n-1}\int_{t_k}^{t_{k+1}}\mathbb{E}[\|{X_2}\|_{_{\L_Q}}^2]d\tau&=\sum_{k=0}^{n-1}\int_{t_k}^{t_{k+1}}\mathbb{E}[\|{T_{n-k}G(u(\tau))}\|_{_{\L_Q}}^2]d\tau\\
&\leq c\sum_{k=0}^{n-1}\int_{t_k}^{t_{k+1}}\mathbb{E}\Big[\big(\frac{\Delta t+h^2}{t_{n-k}}\big)^2\big\|G(u(\tau))\big\|_{_{\L_Q}}^2\Big]d\tau.
\end{aligned}
\end{equation*}
We further use \eqref{linearofG}, \eqref{theresult}, and $\Dt=O(h^2)$ to get
\begin{equation*}\label{part3ine}
	\begin{aligned}
		\| \theta_3^{2}\|_{_{L^2(\Omega,L^2(D))}}^2\leq c\frac{(\Delta t+h^2)^2}{\Dt}\sum_{k=0}^{n-1}\frac{1}{(n-k)^2}
		\leq c\frac{(\Delta t+h^2)^2}{\Dt}\leq c\Dt.
	\end{aligned}
\end{equation*}
For the part $\theta_3^{3}$, employing $O(h^2)=O(J^{-\gamma})$, $\| S^{n-k}_{h,\Delta t}\|_{_{\mathcal{L}({L^2(D)})}}\leq1$ ($\mathbb{P}$-a.s.), and \eqref{f618} gives
\begin{equation*}\label{iii3}
	\begin{aligned}
		\| \theta_3^{3}\|_{_{L^2(\Omega,L^2(D))}}^2
		=\sum_{k=0}^{n-1}\int_{t_k}^{t_{k+1}}\mathbb{E}[\|{X_3}\|_{_{\L_Q}}^2]d\tau
		&\leq c\sum_{k=0}^{n-1}\int_{t_k}^{t_{k+1}}\big(|\tau-t_k|^{\frac{1}{2}}+J^{-\gamma}\big)^2d\tau
		\\&\leq c\sum_{k=0}^{n-1}\int_{t_k}^{t_{k+1}}\big(|\tau-t_k|^{\frac{1}{2}}+h^2\big)^2d\tau
		\\&\leq c(\Delta t^{\frac{1}{2}}+h^2)^2.
	\end{aligned}
\end{equation*}
The last part $\theta_3^{4}$ can be estimated by using It\^o isometry, $\| S^{n-k}_{h,\Delta t}\|_{_{\mathcal{L}({L^2(D)})}}\leq1$ ($\mathbb{P}$-a.s.), and \eqref{f617}: 
\begin{equation*}\label{iii4}
	\begin{aligned}
		\| \theta_3^{4}\|_{_{L^2(\Omega,L^2(D))}}^2=\sum_{k=0}^{n-1}\int_{t_k}^{t_{k+1}}\mathbb{E}[\|{X_4}\|_{_{\L_Q}}^2]d\tau&=\sum_{k=0}^{n-1}\int_{t_k}^{t_{k+1}}\mathbb{E}\Big[\big\| S_{h,\Delta t}^{n-k}\mathcal{P}_{h}\big(G(u(t_k))\mathcal{P}_J^w-G(u^k_{h})\mathcal{P}_J^w\big)\big\|_{_{\L_Q}}^2\Big]d\tau\\&\leq c\sum_{k=0}^{n-1}\int_{t_k}^{t_{k+1}}\mathbb{E}\big[\| u(t_k)-u^k_{h}\|_{_{L^2(D)}}^2\big]d\tau
		\\&= c\sum_{k=0}^{n-1}\| u(t_k)-u^k_{h}\|^2_{_{L^2(\Omega,L^2(D))}}\Delta t.
	\end{aligned}
\end{equation*}
Finally we combine all above estimates and keep only the leading order to conclude.
\end{proof}

Thanks to the results established in the previous lemmas, we are now in a position to derive 
the full discretization error bound, which is stated in the following theorem.
\begin{theorem} \label{the61}
Let $u$ be the mild solution defined in \eqref{for17}, and $u_h^n$ be the numerical solution of \eqref{f53}.
Then under the assumptions stated in Lemmas \ref{lem63}--\ref{lem65}, 
there exists a constant $c$ independent of $\Dt$ and $h$, such that 
	\begin{equation*}\label{conclusion}
		\| u(t_n)-u^n_{h}\|_{_{L^2(\Omega,L^2(D))}}\leq c\big(\Delta t^{\frac{1}{2}}+(\Delta t+h^2)\ln(\Delta t^{-1})\big), \ \ n=1,\dots,N.
	\end{equation*}
\end{theorem}
\begin{proof}
It follows from \eqref{errdec}, Lemmas \ref{lem63}-\ref{lem65}, and the triangle inequality: 
\bex
\| u(t_n)-u^n_{h}\|^2_{_{L^2(\Omega,L^2(D))}}\leq c\Big[\big(\Delta t^{\frac{1}{2}}+(\Delta t+h^2)\ln(\Delta t^{-1})\big)^2+\sum_{k=0}^{n-1}\| u(t_k)-u^k_{h}\|^2_{_{L^2(\Omega,L^2(D))}}\Delta t\Big],\ n=1,\dots,N.
\eex
Then the discrete Gronwall inequality yields
\bex
\| u(t_n)-u^n_{h}\|^2_{_{L^2(\Omega,L^2(D))}}\leq c\big(\Delta t^{\frac{1}{2}}+(\Delta t+h^2)\ln(\Delta t^{-1})\big)^2,\ \ n=1,\cdots, N.
\eex
This ends the proof.               
\end{proof}
\begin{remark}
Notice that the term $\Delta t^{\frac{1}{2}}$ dominates the term $\Delta t \ln(\Delta t^{-1})$, the estimate given in Theorem \ref{the61} 
can be simplified by
\bex
\| u(t_n)-u^n_{h}\|_{_{L^2(\Omega,L^2(D))}}\leq c\big(\Delta t^{\frac{1}{2}}+ h^2\ln(\Delta t^{-1})\big).
\eex
Also notice that $\Delta t^{-\varepsilon_0}$ dominates $\ln(\Delta t^{-1})$ 
for arbitrarily small $\varepsilon_0>0$, we have
\bex
\| u(t_n)-u^n_{h}\|_{_{L^2(\Omega,L^2(D))}}\leq c(\Delta t^{\frac{1}{2}}+ h^2\Delta t^{-\varepsilon_0})
\eex
or, since $\Delta t=O(h^2)$, 
\bex
\| u(t_n)-u^n_{h}\|_{_{L^2(\Omega,L^2(D))}}\leq c(\Delta t^{\frac{1}{2}}+ h^{2-\varepsilon_0}).
\eex
\end{remark}
\section{Numerical results}\label{numericalresult}

Several numerical experiments are presented in this section to validate our theoretical estimates and show the effect of stochastic factors on numerical solutions. We start by testing the convergence orders of time and space. 
\begin{example}[Accuracy Test]\label{example1}
	{\rm We take the stochastic Allen-Cahn (AC) equation with random diffusion coefficient field and multiplicative force noise as a numerical example to test the temporal and spatial convergence orders. The underlying equation is expressed as:
		\begin{equation}
			\begin{aligned}\label{ACeq}
				du(x,t)&=\varepsilon\partial_x\big(e^{z(x,\omega)}\partial_xu\big)dt+(u-u^3)dt+G(u)dW(x,t),\ 0<{t}<T,\ x\in D,\\
				u(x,t)&=0,\quad 0\leq t\leq T,\ x\in\partial{D}, \\
				u(x,0)&=u_0(x),\ x\in \bar{D},
			\end{aligned}
		\end{equation}
	where $z(x,\omega)$ is a W-M Gaussian random field with mean-zero and covariance function $c_q(x)$, and $W(x,t)$ is a $H^{\gamma}_0$-valued Wiener process 
 defined by
\begin{align}
W(t,x)=\sum\limits_{j=1}^\infty \sqrt{q_j} \sin(j\pi x)\beta_j(t),
	\end{align}
where $q_j=\mathcal{O}(j^{-(2\gamma+1+\epsilon)})$ with arbitrary small positive $\epsilon$.
	
	 The strong convergence rate in space and time is measured in terms of mean-square approximation errors at the endpoint $T=0.1$, 
	 caused by the spatial and temporal discretizations. The expected value of error is approximated by computing the mean of 100 samples. Note that the exact solution of the problem \eqref{ACeq} is unknown, and we will use the reference solution computed in the fine space-time mesh size as an approximation to the exact solution. If we denote by $u^{\rm{ref}}_j$ the reference solution of the $j$-th sample of the exact solution $u(T)$, and denote by $u_{j,h}^N$ the value of the $j$-th sample of the fully discrete numerical solution $u_h^N$. Then the mean-square error $\|u(T)-u^N_h\|_{_{L^2(\Omega,L^2(D))}}$ is approximately calculated by
	\bex
	\|u(T)-u^N_h\|_{_{L^2(\Omega,L^2(D))}}\approx\Big(\frac{1}{100}\sum_{j=1}^{100}\|u^{\rm{ref}}_j-u^N_{j,h}\|_{_{L^2(D)}}^2\Big)^\frac{1}{2}=:u_{\rm{error}}.
	\eex
	
	 We first test the time accuracy with different nonlinear terms $G$. Take the numerical solution computed by spatial mesh $h=1/128$ and time step size $\Delta t=10^{-6}$ as the reference solution for every sample. The approximation error $u_{\rm{error}}$ under different time step size is calculated by taking $u_0(x)=\sin(2\pi x)$,  $\varepsilon=10^{-3}$, $\gamma=1$ and $q=2$. Table \ref{time1} and Table \ref{time2} respectively show the results for the cases where $G(u)=(1-u^2)/2$ and $G(u)=u/2$, from which we observe that this is as predicted by the theory. 
		
		Next, we test the spatial accuracy. Now take the numerical solution computed by  $h=1/512$ and $\Delta t= 10^{-6}$ as the reference solution for every sample. We compute the approximation error $u_{\rm{error}}$ under different spatial mesh size by taking $u_0(x)=\sin(2\pi x)$, $\varepsilon=10^{-3}$, $\gamma=1$ and $q=2$ again.
	    Table \ref{space1} and Table \ref{space2} separately shows the relevant error and spatial convergence order for $G(u)=(1-u^2)/2$ and $G(u)=u/2$, which is also consistent with the theoretical result. 
}
\end{example} 
\begin{table}[htbp]
	\begin{minipage}{0.49\linewidth}
		\centering
		\caption{Time accuracy test}
		\label{time1}
		\begin{tabular}{l|ll}
			\hline
			$\Delta t$ & $u_{\rm{error}}$ & Order \\
			\hline
			1.00E-2 & 3.75E-3 & --  \\
			\hline
			5.00E-3 & 2.66E-3 & 0.49\\
			\hline
			2.50E-3 & 1.81E-3 & 0.55 \\
			\hline
			1.25E-3 & 1.34E-3 & 0.43\\
			\hline
			6.25E-4 & 9.32E-4 & 0.52 \\
			\hline
		\end{tabular}
	\end{minipage}
	\begin{minipage}{0.49\linewidth}
		\centering
		\caption{Time accuracy test}
		\label{time2}
		\begin{tabular}{l|ll}
			\hline
			$\Dt$ & $u_{\rm{error}}$ & Order \\
			\hline
			1.00E-2 & 5.23E-3 & --  \\
			\hline
			5.00E-3 & 3.90E-3 & 0.42\\
			\hline
		    2.50E-3 & 2.67E-3 & 0.55 \\
			\hline
			1.25E-3 & 1.83E-3 & 0.55\\
			\hline
			6.25E-4 & 1.33E-3 & 0.46\\
			\hline
		\end{tabular}
	\end{minipage}
\end{table}

\begin{table}[htbp]
	\begin{minipage}{0.49\linewidth}
		\centering
		\caption{Spatial accuracy test}
		\label{space1}
		\begin{tabular}{l|ll}
			\hline
			$h$ & $u_{\rm{error}}$ & Order \\
			\hline
			1/16 & 1.28E-3 & --  \\
			\hline
			1/32 & 3.77E-3 & 1.76\\
			\hline
			1/64 & 1.17E-3 & 1.69 \\
			\hline
			1/128 & 3.61E-4 & 1.70\\
			\hline
			1/256 & 9.74E-5 & 1.89 \\
			\hline
		\end{tabular}
	\end{minipage}
	\begin{minipage}{0.49\linewidth}
		\centering
		\caption{Spatial accuracy test}
		\label{space2}
		\begin{tabular}{l|ll}
			\hline
			$h$ & $u_{\rm{error}}$ & Order \\
			\hline
			1/16 & 1.36E-2 & --  \\
			\hline
			1/32 & 4.09E-3 & 1.73\\
			\hline
			1/64 & 1.33E-3 & 1.62 \\
			\hline
			1/128 & 4.20E-4 & 1.67\\
			\hline
			1/256 & 1.11E-4 & 1.91\\
			\hline
		\end{tabular}
	\end{minipage}
\end{table}


\begin{example}[Phenomenon comparison]\label{example2}
	{\rm In this example, the time evolution of the numerical solution of the stochastic AC equation shown in \eqref{ACeq} is compared to that of the deterministic AC equation to show the effect of random perturbations, where the deterministic version is expressed by: 
		\begin{align*}
			{u}_t(x,t)&=\varepsilon\partial_{xx}{u}+{u}-{u}^3,\  0<t<T,\ x\in D,\\
			u(x,t)&=0,\quad 0\leq t\leq T,\ x\in\partial{D}, \\
			u(x,0)&=u_0(x),\ x\in \bar{D}.
		\end{align*}
	   We first show the effect of random field $a(x,w)=\varepsilon e^{z(x,w)}$ on the numerical solution in the absence of the nonlinear term $G$ (i.e., $G(u)=0$), where $z(x,\omega)$ is a mean-zero Gaussian random field with covariance function $c_q(x)$. Given a sample point, by taking $u_0=\sin(4 \pi x)$, $T=0.1$, $\Delta t=10^{-5}$, $h=1/128$ and $\varepsilon=10^{-2}$, we plot in Figure \ref{Fig:rand-evolution} the contour maps of the numerical solution under different cases, where figure (a) represents the deterministic case and figures (b) and (c) denote the random case with $q=0.1$ and $q=2$, respectively. Compared to the deterministic model, it is seen from figures (b) and (c) that the random diffusion coefficient makes the diffusion process  uncertain. Moreover, it's known that the larger the parameter $q$, the more regular  the random field $z(x,\omega)$ \cite{LORD2014}, which results in the diffusion process shown in figure (c) is more uniform than that in figure (b).
	
	Then we give a demonstration of the case with both random diffusion coefficients as well as multiplicative force noise. Given a sample point, by taking $u_0=\sin(4 \pi x)$, $T=4$, $\Delta t=10^{-4}$, $h=1/128$, $q=2$, $\varepsilon=10^{-5}$ and $G(u)=\frac{1}{2}(1-u^2)$, we plot the time evolution of the numerical solution in Figure \ref{Fig:randnoise-evolution}, where figure (a) denotes the deterministic model and figures (b) and (c) represent the case with  $\gamma=0.5$ and $\gamma=1$, respectively. Compared to figure (a), it can be seen from figures (b) and (c) that there are small-scale structures resulted from noise, which are not present in the deterministic model. Noise plays a significant role, it changes the properties of the solutions. Notably, the static kink corresponding to the deterministic model varies greatly after the incorporation of noise and random diffusion coefficient fields. The kinks can interact, even annihilate each other, and some new kinks may arise. One more thing to point out that the larger the regularity parameter $\gamma$, the smoother the noise and the smaller the kink variation, which seems to be observed between figures (b) and (c).
	}
\end{example} 
\begin{figure}[htbp]
	\centering
	\subfigure[]{
		\includegraphics[width=0.31\linewidth]{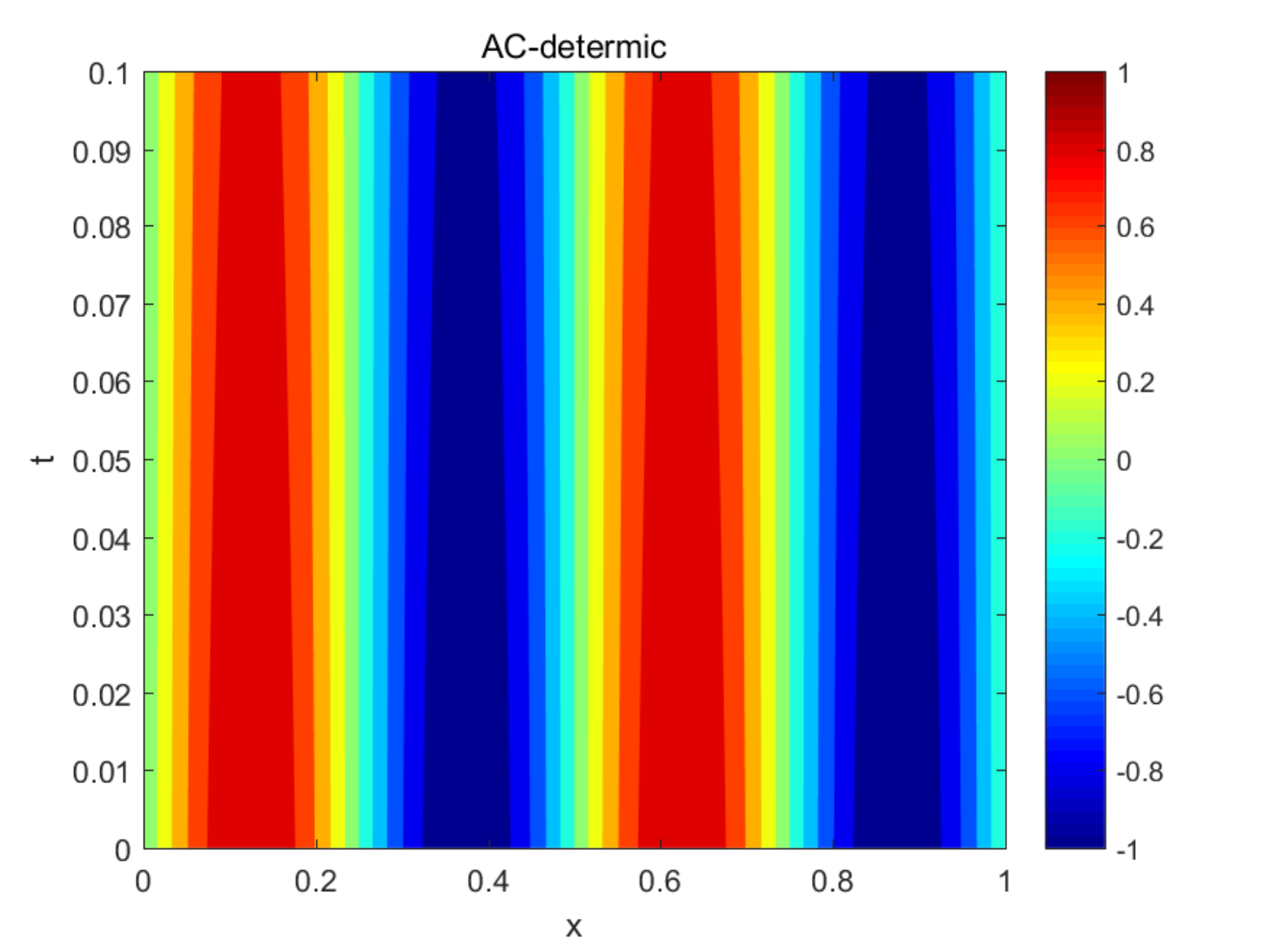}
	}
	\subfigure[]{
			\includegraphics[width=0.31\linewidth]{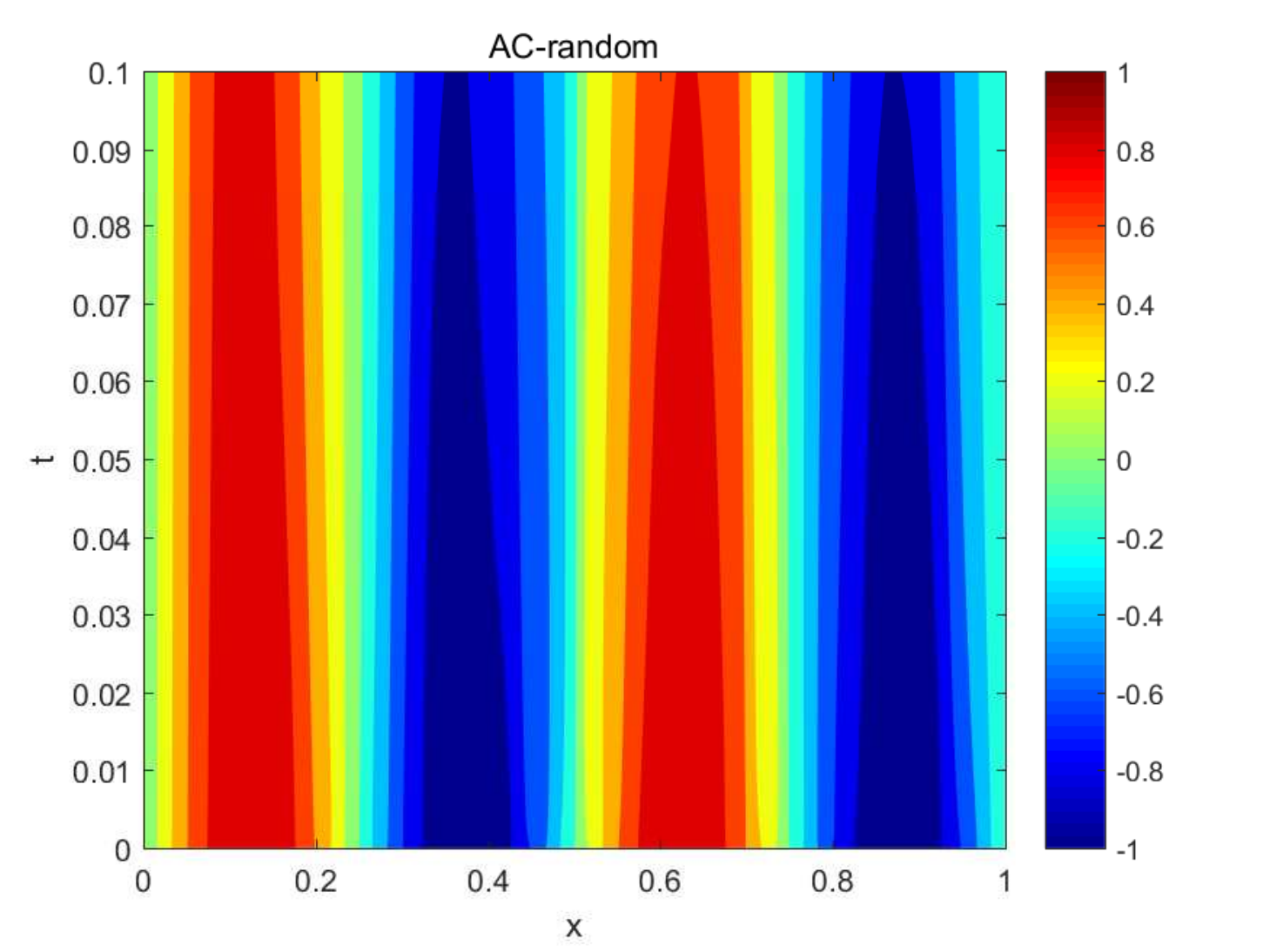}
		} 
	\subfigure[]{
		\includegraphics[width=0.31\linewidth]{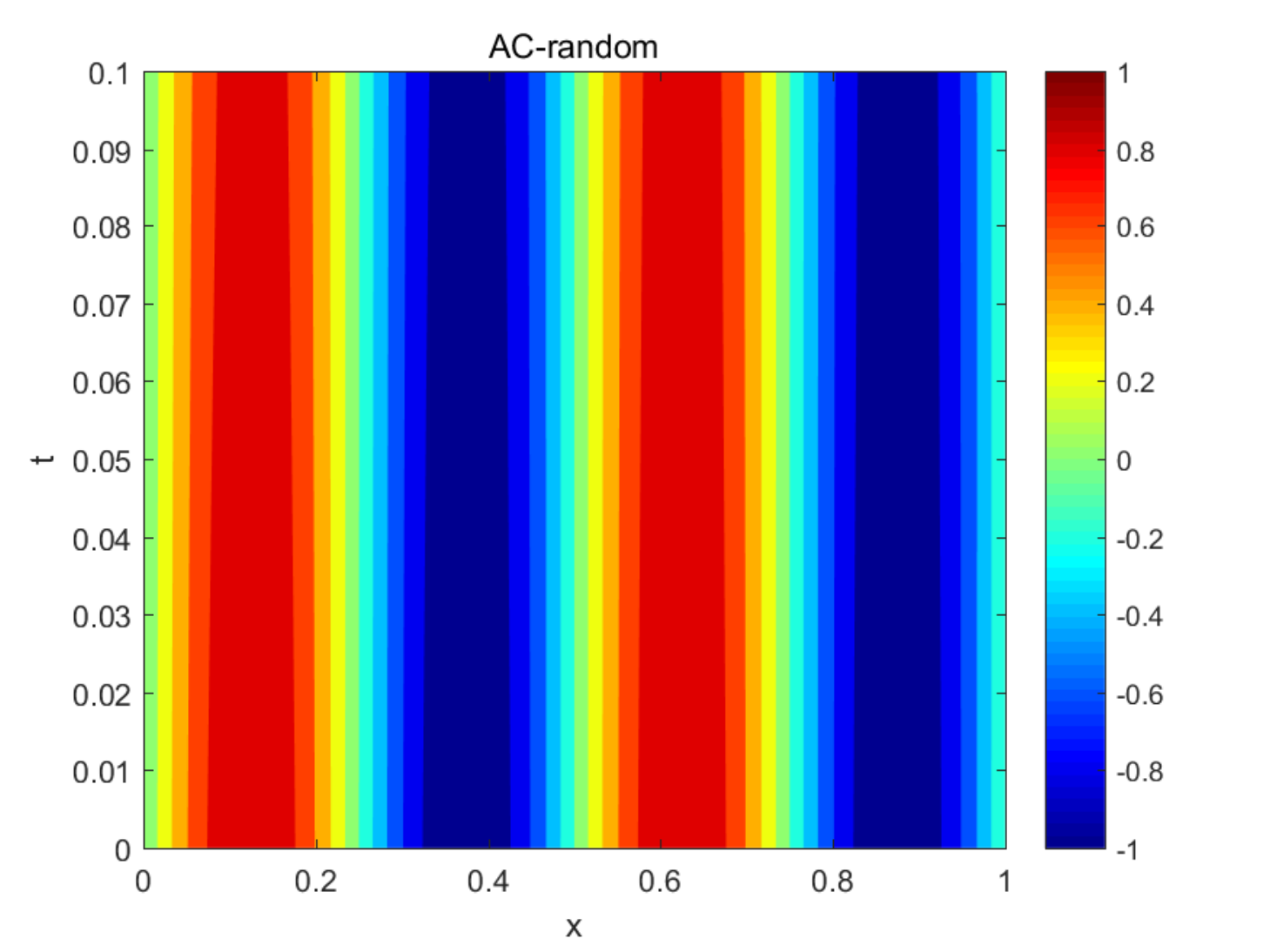}
	}
\caption{Time evolution of numerical solution with different diffusion coefficients. (a): deterministic case, (b): random case with $q=0.1$, (c): random case with $q=2$.}
\label{Fig:rand-evolution}
\end{figure}
\begin{figure}[htbp]
	\centering
	\subfigure[]{
		\includegraphics[width=0.31\linewidth]{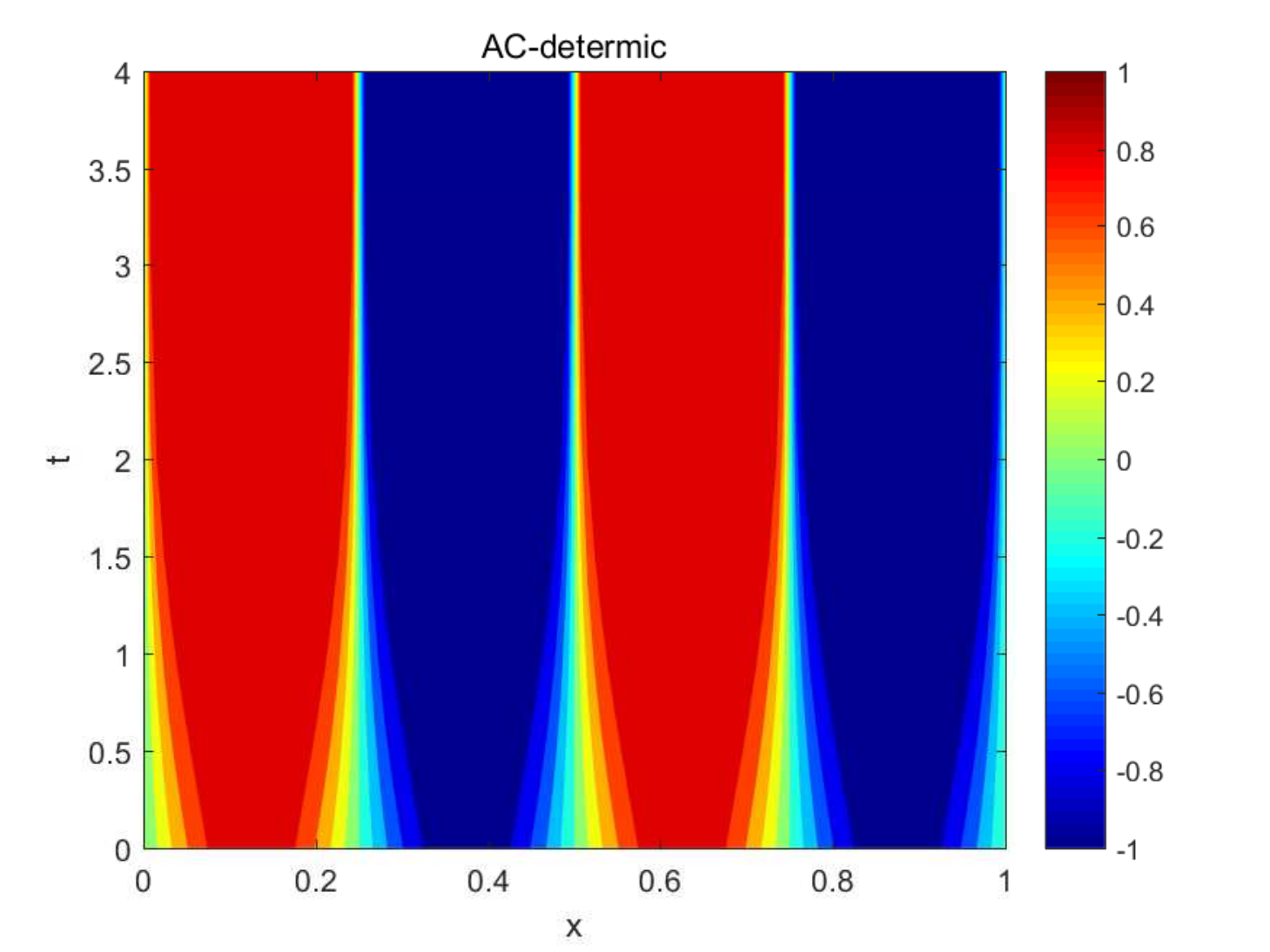}
	}
	\subfigure[]{
		\includegraphics[width=0.31\linewidth]{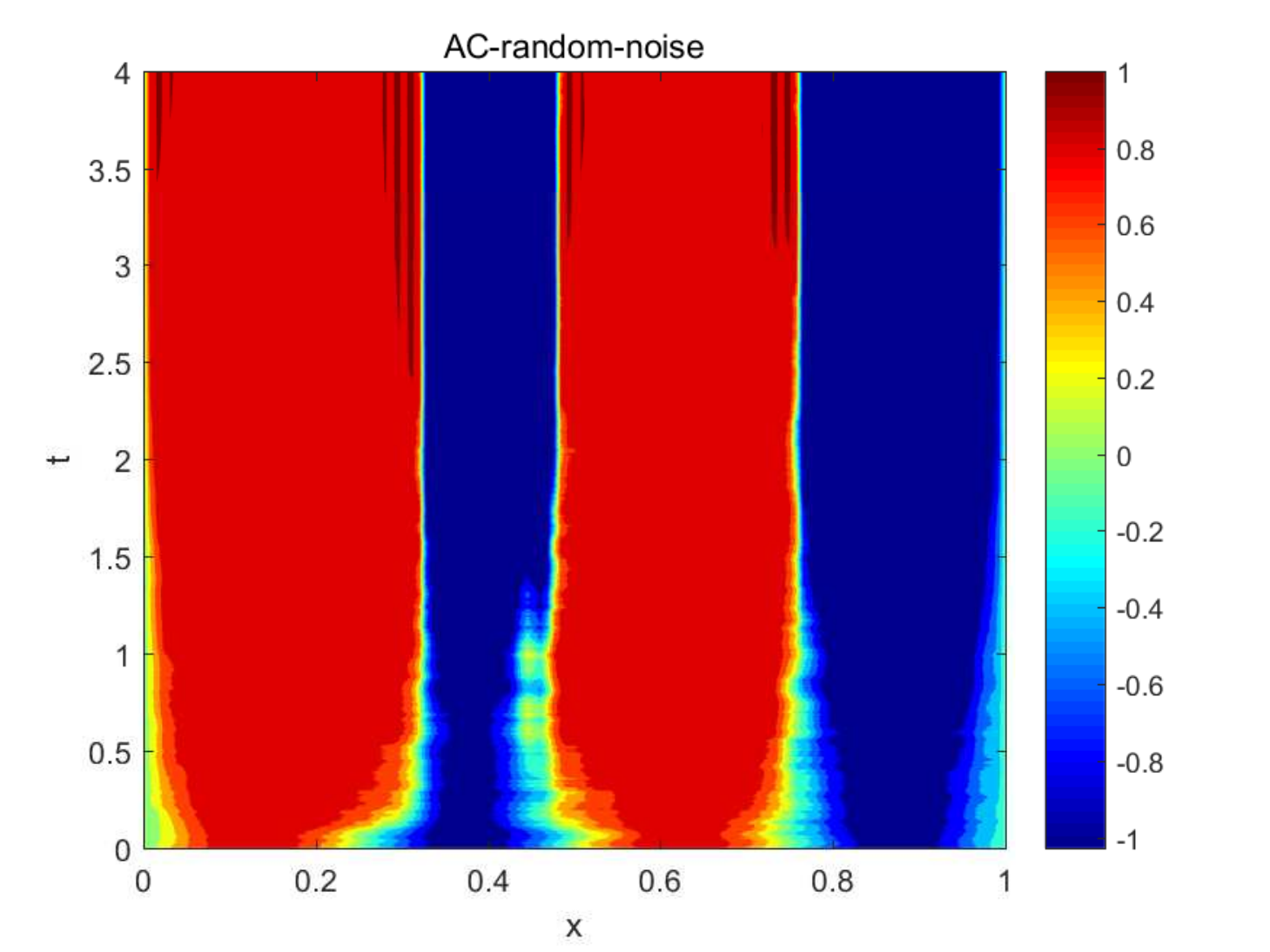}
	} 
	\subfigure[]{
		\includegraphics[width=0.31\linewidth]{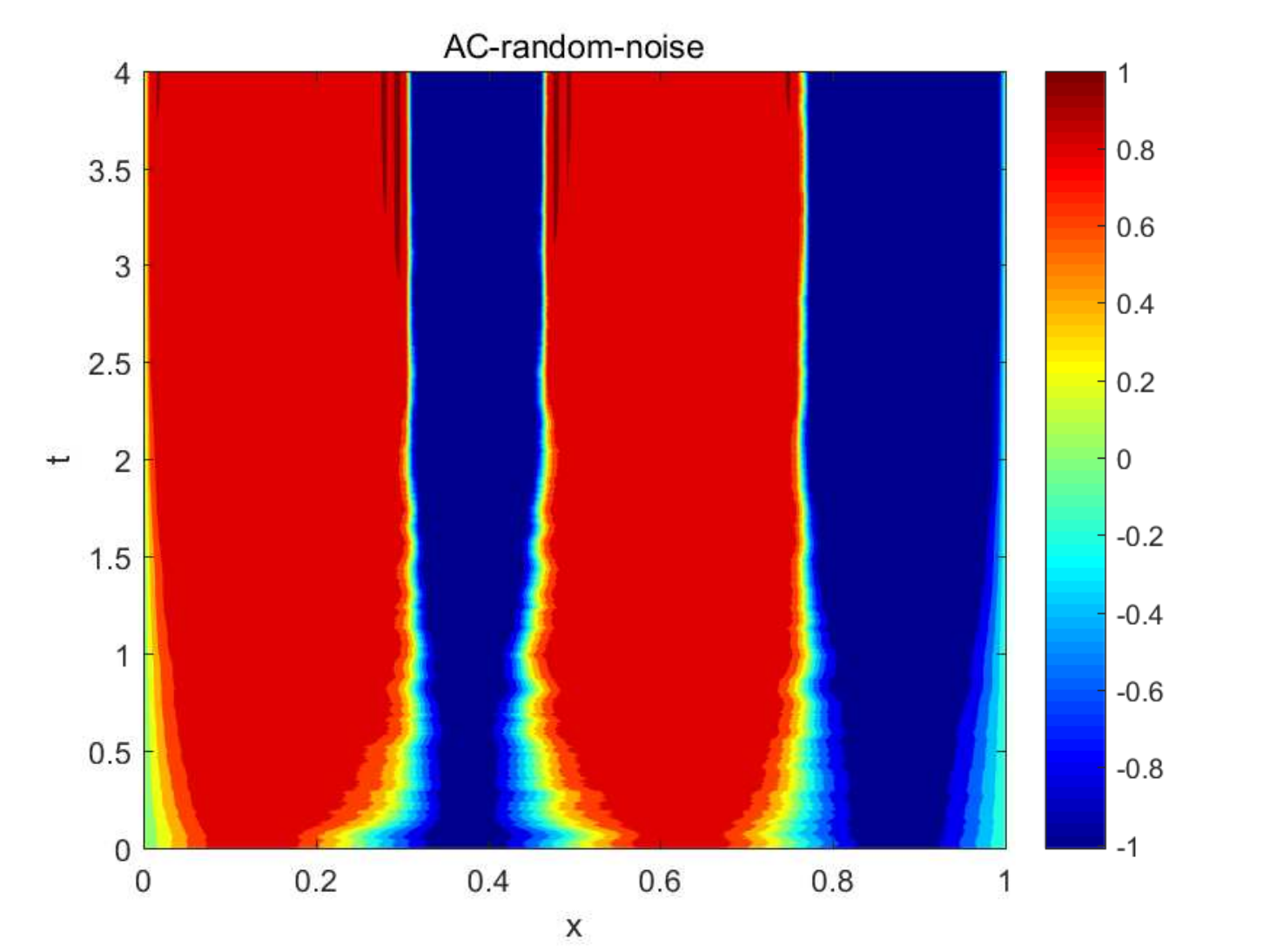}
	}
	\caption{Time evolution of numerical solution with different noise. (a): deterministic case, (b): random case with $\gamma=0.5$, (c): random case with $\gamma=1$.}
\label{Fig:randnoise-evolution}
\end{figure}

\bibliographystyle{plain}
\bibliography{ref}

\end{document}